\documentclass[10pt]{article}
\usepackage{epsfig}
\usepackage{amsmath}
\usepackage{amssymb}
\usepackage{graphicx}
\usepackage[latin1]{inputenc}
\usepackage{amsthm}
\usepackage[french,english]{babel}
{\catcode `\@=11 \global\let\AddToReset=\@addtoreset}
\AddToReset{equation}{section}

\newtheorem{lemma}{\bf Lemma}[section]

\newtheorem{property}{Property}[section]

\newtheorem{@definition}{\sc Definition}[section]

\newtheorem{@remark}{\sc Remark}[section]

\newtheorem{@example}{\sc Example}[section]

\newcommand{\beqn}{\begin{displaymath}}
\newcommand{\eeqn}{\end{displaymath}}
\newcommand{\beq}{\begin{equation}}  
\newcommand{\eeq}{\end{equation}}

\def\mathsf{\bf}
\def\N{\mathbb{N}}

\def\R{\mathbb{R}}
\def\Z{\mathbb{Z}}

\def\E{\mathrm E}

\def\text{\mbox}

\def\1{{\bf 1}}

\newcommand{\Cov}{\mbox{Cov}}

\newcommand{\Loi}{\mathcal{L}}

\def\limiteloiN{\renewcommand{\arraystretch}{0.5}
\begin{array}[t]{c}
\stackrel{{\Loi}}{\longrightarrow} \\
{\scriptstyle N\rightarrow\infty}
\end{array}\renewcommand{\arraystretch}{1}}

\def\limiteloiNm{\renewcommand{\arraystretch}{0.5}
\begin{array}[t]{c}
\stackrel{{\Loi}}{\longrightarrow} \\
{\scriptstyle [N/m]\wedge m \rightarrow\infty}
\end{array}\renewcommand{\arraystretch}{1}}

\def\limiteprobaN{\renewcommand{\arraystretch}{0.5}
\begin{array}[t]{c}
\stackrel{{\cal P}}{\longrightarrow} \\
{\scriptstyle N \rightarrow\infty}
\end{array}\renewcommand{\arraystretch}{1}}

\def\limitem{\renewcommand{\arraystretch}{0.5}
\begin{array}[t]{c}
\stackrel{}{\longrightarrow} \\
{\scriptstyle m\rightarrow\infty}
\end{array}\renewcommand{\arraystretch}{1}}

\def\limiteNm{\renewcommand{\arraystretch}{0.5}
\begin{array}[t]{c}
\stackrel{}{\longrightarrow} \\
{\scriptstyle N,~m,~\frac N m \rightarrow\infty}
\end{array}\renewcommand{\arraystretch}{1}}

\def\limiteN{\renewcommand{\arraystretch}{0.5}
\begin{array}[t]{c}
\stackrel{}{\longrightarrow} \\
{\scriptstyle N\rightarrow\infty}
\end{array}\renewcommand{\arraystretch}{1}}

\def\limitet{\renewcommand{\arraystretch}{0.5}
\begin{array}[t]{c}
\stackrel{}{\longrightarrow} \\
{\scriptstyle t\rightarrow\infty}
\end{array}\renewcommand{\arraystretch}{1}}

\def\limitet0{\renewcommand{\arraystretch}{0.5}
\begin{array}[t]{c}
\stackrel{}{\longrightarrow} \\
{\scriptstyle t\rightarrow 0}
\end{array}\renewcommand{\arraystretch}{1}}

\newtheorem{thm}{Theorem}
\newtheorem{rem}{Remark}

\newtheorem{prop}{Proposition}

\def\Cov{\mathrm{Cov}}

\begin{document}

\title{\bf Adaptive estimator of the memory parameter and goodness-of-fit test using a multidimensional increment ratio statistic}

\author{\centerline{Jean-Marc Bardet and Béchir Dola} \\
\small {\tt bardet@univ-paris1.fr}~~,~~ \small {\tt bechir.dola@malix.univ-paris1.fr}\\
~\\
SAMM, Université Panthéon-Sorbonne (Paris I), 90 rue de Tolbiac, 75013 Paris, FRANCE}
\maketitle
\begin{abstract}
The increment ratio (IR) statistic was first defined and studied in Surgailis {\it et al.} (2007) for estimating the memory parameter either of a stationary or an increment stationary Gaussian process. Here three extensions are proposed in the case of stationary processes. Firstly, a multidimensional central limit theorem is established for a vector composed by several IR statistics. Secondly, a goodness-of-fit $\chi^2$-type test can be deduced from this theorem. Finally, this theorem allows to construct adaptive versions of the estimator and test which  are studied in a general semiparametric frame. The adaptive estimator of the long-memory parameter is proved to follow an oracle property. Simulations attest of the interesting accuracies and robustness of the estimator and test, even in the non Gaussian case.
\end{abstract}
\begin{quote}
\end{quote}

\vskip1cm

\section{Introduction}
After almost thirty years of intensive and numerous studies, the long-memory processes form now an important topic of the time series study (see for instance the book edited by Doukhan {\it et al}, 2003). The most famous long-memory stationary time series are the fractional Gaussian noises (fGn) with Hurst parameter $H$ and FARIMA$(p,d,q)$ processes. For both these time series, the spectral density $f$ in $0$  follows a power law: $f(\lambda) \sim C\, \lambda^{-2d}$ where $H=d+1/2$ in the case of the fGn. In the case of long memory process $d \in (0,1/2)$ but a natural expansion to $d \in (-1/2,0]$ (short memory) implied that $d$ can be considered more generally as a memory parameter. \\
There are a lot of statistical results relative to the estimation of this memory parameter $d$. First and main results in this direction have been obtained for parametric models with the essential articles of Fox and Taqqu (1986) and Dahlhaus (1989) for Gaussian time series, Giraitis and Surgailis (1990) for linear processes and Giraitis and Taqqu (1999) for non linear functions of Gaussian processes.  \\
However parametric estimators are not really robust and can induce no consistent estimations. Thus, the research  is now rather focused on semiparametric estimators of the memory parameter. Different approaches were considered: the famous R/S statistic (see Hurst, 1951), the log-periodogram estimator (studied firstly by Geweke and Porter-Hudack, 1983, notably improved by Robinson, 1995a, and Moulines and Soulier, 2003), the local Whittle estimator (see Robinson, 1995b) or the wavelet based estimator (see Veitch {\it et al}, 2003, Moulines {\it et al}, 2007 or Bardet {\it et al}, 2008). All these estimators require the choice of an auxiliary parameter (frequency bandwidth, scales, etc.) but adaptive versions of these estimators are generally built for avoiding this choice. In a general semiparametric frame, Giraitis {\it et al} (1997) obtained the asymptotic lower bound for the minimax risk in the estimation of $d$, expressed as a function of the second order parameter of the spectral density expansion around $0$. Several adaptive semiparametric estimators are proved to follow an oracle property up to multiplicative logarithm term. But simulations (see for instance Bardet {\it et al}, 2003 or 2008) show that the most accurate estimators are local Whittle, global log-periodogram and wavelet based estimators. \\
~\\
In this paper, we consider the IR (Increment Ratio) estimator of long-memory parameter (see its definition in the next Section) for Gaussian time series recently  introduced in Surgailis {\it et al.} (2007) and we propose three extensions. Firstly, a multivariate central limit theorem is established for a vector of IR statistics with different ``windows'' (see Section \ref{Sec2}) and this induces to consider a pseudo-generalized least square estimator of the parameter $d$. Secondly, this multivariate result allows us to define an adaptive estimator of the memory parameter $d$ based on IR statistics: an ``optimal'' window is automatically computed (see {Section} {\ref{Adapt}}). This notably improves the results of Surgailis {\it et al.} (2007) in which the choice of $m$ is either theoretical (and cannot be applied to data) or guided by empirical rules without justifications. Thirdly, an adaptive goodness-of-fit test is deduced and its convergence to a chi-square distribution is established (see Section {\ref{Adapt}}).\\ 
In Section \ref{Simu}, several Monte Carlo simulations are realized for optimizing the adaptive estimator and exhibiting the theoretical results. Then some numerical comparisons are made with the $3$ semiparametric estimators previously mentioned (local Whittle, global log-periodogram and wavelet based estimators) and the results are even better than the theory seems to indicate: as well in terms of convergence rate than in terms of robustness (notably in case of trend or seasonal component), the adaptive IR estimator and goodness-of-fit test provide efficient results. Finally, all the proofs are grouped in {Section} {\ref{proofs}}.

\section{The multidimensional increment ratio statistic and its statistical applications} \label{Sec2}
Let $X=(X_k)_{k\in \N}$ be a Gaussian time series satisfying the following Assumption $S(d,\beta)$:\\
~\\
{\bf Assumption $S(d,\beta)$:~} {\em There exist $\varepsilon > 0$, $c_0>0$, $c'_0>0$  and $c_1 \in \R$ such that $X=(X_t)_{t\in \Z}$ is a stationary Gaussian time series having a spectral density $f$ satisfying for all $\lambda \in (-\pi,0)\cup (0,\pi)$}
\begin{eqnarray}\label{AssumptionS}
f(\lambda)= c_{0}|\lambda|^{-2d}+c_{1}|\lambda|^{-2d+\beta} +O\big(|\lambda|^{-2d+\beta+\varepsilon}\big)
\quad \mbox{and}\quad |f'(\lambda)|\leq  c'_0 \, \lambda^{-2d-1}.
\end{eqnarray}
\begin{rem}
Note that here we only consider the case of stationary processes. However, as it was already done in Surgailis {\em et al.} (2007), it could be possible, {\em mutatis mutandis}, to extend our results to the case of processes having stationary increments. 
\end{rem}
\noindent Let $(X_1,\cdots,X_N)$ be a path of $X$. For $m\in \N^*$, define the random variable $IR_{N}(m)$ such as
$$
IR_N(m):=\frac{1}{N-3m}\sum_{k=0}^{N-3m-1}\frac{|(\sum_{t=k+1}^{k+m}X_{t+m}-\sum_{t=k+1}^{k+m}X_{t})+(\sum_{t=k+m+1}^{k+2m}X_{t+m}-\sum_{t=k+m+1}^{k+2m}X_{t})|}{|(\sum_{t=k+1}^{k+m}X_{t+m}-\sum_{t=k+1}^{k+m}X_{t})|+|(\sum_{t=k+m+1}^{k+2m}X_{t+m}-\sum_{t=k+m+1}^{k+2m}X_{t})|}.
$$
From Surgailis {\it et al.} (2007), with $m$ such that $N/m \to \infty$ and $m\to \infty$,
$$
\sqrt{\frac{N}{m}}\big (IR_N(m)-\E IR_N(m)\big )\limiteloiN {\cal N}(0,\sigma^{2}(d)),
$$
where
\begin{eqnarray}
\label{sigma}
\sigma^{2}(d)&:=&2\int_{0}^{\infty}\Cov \Big (\frac{|Z_{d}(0)+Z_{d}(1)|}{|Z_{d}(0)|+|Z_{d}(1)|}\,,\, \frac{|Z_{d}(\tau)+Z_{d}(\tau+1)|}{|Z_{d}(\tau)|+|Z_{d}(\tau+1)|}\Big )d\tau\\
\label{Zd}
\mbox{and}\quad  Z_{d}(\tau)&:=& \frac{1}{\sqrt{|4^{d+0.5}-4|}} \big ( B_{d+0.5}(\tau+2)-2\, B_{d+0.5}(\tau+1)+B_{d+0.5}(\tau) \big )
\end{eqnarray}
with $B_H$ a standardized fractional Brownian motion (FBM) with Hurst parameter $H\in (0,1)$. 
\begin{rem}
This convergence was obtained for Gaussian processes in Surgailis {\it et al.} (2007), but there  also exist results concerning a modified IR statistic applied to stable processes (see Vaiciulis, 2009) with a different kind of limit theorem. We may suspect that it is also possible to extend the previous central limit theorem to long memory linear processes (since a Donsker type theorem with FBM as limit was proved for long memory linear processes, see for instance Ho and Hsing, 1997) but such result requires to prove a non obvious central limit theorem for a functional of multidimensional linear process. Surgailis {\it et al.} (2007) also considered the case of i.i.d.r.v. in the
domain of attraction of a stable law with index $0<\alpha<2 $ and skewness parameter $-1\leq \beta \leq  1$  and concluded that $IR_N(m)$ converges to almost the same limit. Finally, in Bardet and Surgailis (2011) a ``continuous'' version of the IR statistic is considered for several kind of continuous time processes (Gaussian processes, diffusions and Lévy processes). 
\end{rem}
\noindent Now, instead of this univariate IR statistic, define a multivariate IR statistic as follows: let $m_{j}=j\, m,~j=~1,\cdots,p$ with $2\leq p [N/m]-4$, and define the random vector $(IR_{N}(j \, m))_{1\leq j\leq p}$. Thus, $p$ is the number of considered window lengths of this multivariate statistic. In the sequel we naturally extend the results obtained for $m\in \N^*$ to $m\in (0,\infty)$ by the convention: $(IR_{N}(j \, m))_{1\leq j\leq p}=(IR_{N}(j \, [m]))_{1\leq j\leq p}$ (which change nothing to the asymptotic results). \\
We can establish a multidimensional central limit theorem satisfied by $(IR_{N}(j \, m))_{1\leq j\leq p}$:
\begin{property}\label{MCLT}
Assume that Assumption $S(d,\beta)$ holds with $-0.5<d<0.5$ and $\beta>0$. Then
\begin{eqnarray}\label{TLC1}
\sqrt{\frac{N}{m}}\Big (IR_N(j \, m)-\E \big [IR_N(j \, m)\big ]\Big )_{1\leq j \leq p}\limiteloiNm {\cal N}(0, \Gamma_p(d))
\end{eqnarray}
with $\Gamma_p(d)=(\sigma_{i,j}(d))_{1\leq i,j\leq p}$ where for $t\in \R$
\begin{multline}\label{ssigma}
\sigma_{i,j}(d):=\int_{-\infty}^{\infty}\Cov \Big (\frac{|Z^{(i)}_{d}(0)+Z^{(i)}_{d}(i)|}{|Z^{(i)}_{d}(0)|+|Z^{(i)}_{d}(i)|}\,,\, \frac{|Z^{(j)}_{d}(\tau)+Z^{(j)}_{d}(\tau+j)|}{|Z^{(j)}_{d}(\tau)|+|Z^{(j)}_{d}(\tau+j)|}\Big )d\tau\\
\mbox{and}\quad Z_{d}^{(j)}(\tau)=\frac 1 {\sqrt{|4^{d+0.5}-4|}} \,  \big ( B_{d+0.5}(\tau+2j)-2B_{d+0.5}(\tau+j)+B_{d+0.5}(\tau) \big ).
\end{multline}
\end{property}
\noindent The proof of this property as well as all the other proofs are given in {Appendix}. Moreover we will assume in the sequel that $\Gamma_p(d)$ is a definite positive matrix for all $d \in (-0.5,0.5)$.\\
\begin{rem}
Note that Assumption $S(d,\beta)$ are a little stronger than the conditions required in Surgailis {\it et al.} (2007) where $f$ is supposed to satisfy $f(\lambda)=c_0|\lambda|^{-2d} + O(|\lambda|^{-2d+\beta})$ and $|f'(\lambda)|\leq  c'_0 \, \lambda^{-2d-1}$. Note that Property \ref{MCLT} and following Theorem \ref{cltnada} and Proposition \ref{testnada} are as well checked under these assumptions of Surgailis {\it et al.} (2007) even if $\beta\geq 2d+1$ (case which is not consider in their Theorem 2.4). However our automatic procedure for choosing an adaptive scale $\widetilde m_N$ requires to specify the second order of the expansion of $f$ and we prefer to already give results under such assumption.
\end{rem}
\noindent As in Surgailis {\it et al.} (2007), for $r \in (-1,1)$, define the function $\Lambda(r)$ by
\begin{eqnarray}\label{Lambdad}
\Lambda(r):=\frac{2}{\pi}\, \arctan\sqrt{\frac{1+r}{1-r}}+\frac{1}{\pi}\, \sqrt{\frac{1+r}{1-r}}\log(\frac{2}{1+r}).
\end{eqnarray}
and for $d \in (-0.5,1.5)$ let
\begin{eqnarray}\label{Lambda0d}
\Lambda_0(d)&:=&\Lambda(\rho(d))\quad \mbox{where} \quad \rho(d):=\frac{4^{d+1.5}-9^{d+0.5}-7}{2(4-4^{d+0.5})}.
\end{eqnarray}
The function $d \in (-0.5,1.5) \to \Lambda_0(d)$ is a ${\cal C}^\infty$ increasing function. Now, Property \ref{devEIR} (see in Section \ref{proofs}) provides the asymptotic behavior of $\E [IR(m)]$ when $m\to \infty$, which is  $\E [IR(m)]\sim \Lambda_0(d)+C m^{-\beta}$ if $\beta<2d+1$, $\E [IR(m)]\sim \Lambda_0(d)+C m^{-\beta}\log m$ if $\beta=2d+1$ and $\E [IR(m)]\sim \Lambda_0(d)+ O(m^{-(2d+1)})$ if $\beta>2d+1$ ($C$ is a non vanishing real number depending on $d$ and $\beta$). Therefore by choosing $m$ and $N$ such as $\big (\sqrt {N/m}\big ) m^{-\beta}\to 0$, $\big (\sqrt {N/m} \big )m^{-\beta}\log m \to 0$ and $\big (\sqrt {N/m}\big ) m^{-(2\beta+1)}\to 0$ (respectively) when $m,N\to 
\infty$, the term $\E [IR(jm)]$ can be replaced by $\Lambda_0(d)$ in Property \ref{MCLT}. Then, using the Delta-method with function $(x_i)_{1\leq i \leq p} \mapsto (\Lambda^{-1}_0(x_i))_{1\leq i \leq p}$, we obtain:
\begin{thm}\label{cltnada}
Let $\widehat d_N(j \, m):=\Lambda_0^{-1}\big (IR_N(j \, m)\big )$ for $1\leq j \leq p$. Assume that Assumption $S(d,\beta)$ holds with $-0.5<d<0.5$ and $\beta>0$. Then if $m \sim C\, N^\alpha$ with $C>0$ and $(1+2\beta)^{-1}\vee (4d+3)^{-1}<\alpha<1$ then
\begin{eqnarray}\label{cltd}
\sqrt{\frac{N}{m}}\Big (\widehat d_N(j \, m)-d\Big )_{1\leq j \leq p}\limiteloiN {\cal N}\Big(0,(\Lambda'_0(d))^{-2}\, \Gamma_p(d)\Big ).
\end{eqnarray}
\end{thm}
\begin{rem}
If $\beta<2d+1$, the estimator $\widehat
d_N(m)$ is a semiparametric estimator of $d$ and its asymptotic
mean square error can be minimized with an appropriate sequence $(m_N)$ reaching the well-known minimax rate of convergence
for memory parameter $d$ in this semiparametric setting (see for
instance Giraitis {\em et al.}, 1997). Indeed,
under Assumption $S(d,\beta)$ with $d \in (-0.5,0.5)$ and $\beta>0$ and if
$m_N=[N^{1/(1+2\beta)}]$, then the estimator $\widehat d_N(m_N)$ is rate
optimal in the minimax sense, {\em i.e.}
$$
\limsup _{N \to \infty}\sup_{d\in (-0.5,0.5)} ~~\sup_{f  \in S(d,\beta)}N^{\frac {2\beta}{1+2\beta}} \cdot \E [({\widehat d}_N(m_N)
-d)^2]<\infty.
$$
\end{rem}
\noindent From the multidimensional CLT (\ref{cltd}) a pseudo-generalized least square estimation (LSE) of $d$ is possible by defining the following matrix:
\begin{equation}\label{Sigma}
\widehat \Sigma_N(m):=(\Lambda'_0(\widehat d_N(m))^{-2}\, \Gamma_p(\widehat d_N(m)).
\end{equation}
Since the function $d\in (-0.5,1.5) \mapsto \sigma(d)/\Lambda'(d)$ is ${\cal C}^\infty$ it is obvious that under assumptions of Theorem \ref{cltnada} then
$$
\widehat \Sigma_N(m) \limiteprobaN (\Lambda'_0(d))^{-2}\, \Gamma_p(d).
$$
Then with the vector $J_p:=(1)_{1\leq j \leq p}$ and denoting $J_p'$ its transpose, the pseudo-generalized LSE of $d$ is:
$$
\widetilde d_N(m):=\big (J_p' \big (\widehat \Sigma_N(m)\big )^{-1} J_p \big )^{-1}\, J_p' \,  \big ( \widehat \Sigma_N(m) \big )^{-1} \big (\widehat d_N(m_i) \big ) _{1\leq i \leq p}
$$
It is well known (Gauss-Markov Theorem) that the Mean Square Error (MSE) of $\widetilde d_N(m)$ is smaller or equal than all the MSEs of $\widehat d_N(jm)$, $j=1,\ldots,p$. Hence, we obtain under the assumptions of Theorem \ref{cltd}:
\begin{eqnarray}\label{TLCdtilde}
\sqrt{\frac{N}{m}}\big (\widetilde d_N(m)-d\big ) \limiteloiN {\cal N}\Big(0 \, , \, \Lambda'_0(d)^{-2}\,\big (J_p' \, \Gamma^{-1}_p(d)J_p\big )^{-1}\Big ),
\end{eqnarray}
and $\Lambda'_0(d)^{-2}\big (J_p' \, \Gamma^{-1}_p(d)J_p\big )^{-1} \leq \Lambda'_0(d)^{-2} \sigma^2(d)$. \\
~\\
Now, consider the following test problem: for $(X_1,\cdots,X_n)$ a path of $X$ a Gaussian time series, chose between
\begin{itemize}
\item $H_0$: the spectral density of $X$ satisfies Assumption $S(d,\beta)$ with $-0.5<d<0.5$ and $\beta>0$; 
\item $H_1$: the spectral density of $X$ does not satisfy such a behavior.
\end{itemize}
We deduce from the multidimensional CLT (\ref{cltd}) a $\chi^2$-type goodness-of-fit test statistic defined by:
$$
\widehat T_N(m):=\frac N m \, \big (\widetilde d_N(m)-\widehat d_N(j \, m) \big )'_{1\leq j \leq p}\big (\widehat \Sigma_N(m)\big )^{-1}\big (\widetilde d_N(m)-\widehat d_N(j \,m) \big )_{1\leq j \leq p}.
$$
Then the following limit theorem can be deduced from {Theorem} {\ref{cltnada}}:
\begin{prop}\label{testnada}
Under the assumptions of Theorem \ref{cltnada} then:
$$
\widehat T_N(m) \limiteloiN \chi^2(p-1).
$$
\end{prop}
\section{Adaptive versions of the estimator and goodness-of-fit test}\label{Adapt}
Theorem \ref{cltnada} and Proposition \ref{testnada} are interesting but they require the knowledge of $\beta$ to be used (and therefore an appropriated choice of $m$). We suggest now a procedure (see also Bardet {\em et al.}, 2008) for obtaining a data-driven selection of an optimal sequence $(m_N)$. For $d \in (-0.5,1.5)$ and $\alpha \in (0,1)$, define
\begin{equation}\label{QNdef}
Q_N(\alpha,d):=\big (\widehat d_N(j\, N^\alpha)-d \big )'_{1\leq j \leq p} \big (\widehat \Sigma_N(N^\alpha)\big )^{-1}\big (\widehat d_N(j\, N^\alpha)-d \big )_{1\leq j \leq p}.
\end{equation}
Note that by the previous convention, $\widehat d_N(j\, N^\alpha)=\widehat d_N(j\, [N^\alpha])$ and $\widetilde d_N(N^\alpha)=\widetilde d_N([N^\alpha])$. Thus
$Q_N(\alpha,d)$ corresponds to the sum of the pseudo-generalized squared distance. From previous computations, it is obvious that for a fixed
$\alpha\in (0,1)$, $Q$ is minimized by $\widetilde d_N(N^\alpha)$ and therefore for $0<\alpha<1$ define
$$
\widehat Q_N(\alpha):=Q_N(\alpha,\widetilde d_N(N^{\alpha})).
$$
It remains to minimize $\widehat Q_N(\alpha)$ on $(0,1)$.
However, since $\widehat \alpha_N$ has to be obtained from
numerical computations, the interval $(0,1)$ can be discretized as
follows,
$$
\widehat \alpha_N \in {\cal A}_N=\Big \{\frac {2}{\log
N}\,,\,\frac { 3}{\log N}\,, \ldots,\frac {\log [N/p]}{\log N}
\Big \}.
$$
Hence, if $\alpha \in {\cal A}_N$, it exists $k \in \{2,3,\ldots,
\log [N/p]\}$ such that $k=\alpha \, \log N$.
Consequently,
define $\widehat \alpha_N$ by
\begin{eqnarray*}
\widehat Q_N(\widehat \alpha_N ):=\min_{\alpha \in {\cal A}_N}
\widehat Q_N(\alpha).
\end{eqnarray*}
\noindent From the central limit theorem ({\ref{cltd}}) one deduces the following : 
\begin{prop}\label{hatalpha}
Assume that Assumption $S(d,\beta)$ holds with $-0.5<d<0.5$ and $\beta>0$. Moreover, if $\beta>2d+1$, suppose that $c_0,c_1,c_2,d,\beta$ and $\varepsilon$ are such that Condition (\ref{condbis1}) or (\ref{condbis2}) holds. Then,
$$
\widehat \alpha_N
\limiteprobaN \alpha^*=\frac 1 {(1+2\beta)\wedge (4d+3)}.
$$
\end{prop}
\begin{rem}\label{defAn}
The choice of the set of discretization ${\cal A}_N$ is implied by our proof of convergence of $\widehat \alpha_N$ to $\alpha^*$. If the interval $(0,1)$ is stepped in $N^c$ points, with $c>0$, the used proof cannot
attest this convergence. However $\log N$ may be replaced in the previous expression of ${\cal A}_N$ by any negligible function
of $N$ compared to functions $N^c$ with $c>0$ (for instance, $(\log N)^a$ or $a\log N$ can be used).
\end{rem}
\begin{rem}\label{remcond}
The reference to Condition (\ref{condbis1}) or (\ref{condbis2}) is necessary because our proof of the convergence of $\widehat \alpha_N$ to $\alpha^*$ requires to know the exact convergence rate of $\E[IR_N(N^\alpha)]-\Lambda_0(d)$ when $\alpha<\alpha^*$. When $\beta\leq 2d+1$, since we replaced the conditions on the spectral density of Surgailis {\it et al.} (2007) by a second order condition (Assumption $S(d,\beta)$), this convergence rate can be obtained by computations (see Property \ref{devEIR}). But if $\beta> 2d+1$, we can only obtain $\E[IR_N(N^\alpha)]-\Lambda_0(d)=O(m^{-2d-1})$ under Assumption $S(d,\beta)$: the convergence rate could be slower than $m^{-2d-1}$ and then $\widehat \alpha_N$ could converge to $\alpha'<\alpha^*$ (from the proof of Proposition \ref{hatalpha}). Condition (\ref{condbis1}) and (\ref{condbis2}), which are not very strong, allow to obtain a first order bound for $\E[IR_N(N^\alpha)]-\Lambda_0(d)$ (see Proposition \ref{devEIR2}) and hence to prove $\widehat \alpha_N \limiteprobaN \alpha^*$.
\end{rem}
\noindent
From a straightforward application of the proof of Proposition \ref{hatalpha}, the asymptotic
behavior of $\widehat a_N$ can be specified, that is,
\begin{eqnarray}\label{hathatD}
\Pr \Big ( \frac {N^{\alpha^*}} {(\log N) ^\lambda}\leq
N^{\widehat \alpha_N} \leq N^{\alpha^*}\cdot ( \log N )^\mu \Big
)\limiteN 1,
\end{eqnarray}
for all positive real numbers $\lambda$ and $\mu$ such that $\lambda
> \frac {2\alpha^*} {(p-2)(1-\alpha^*)}$ and $\mu >\frac {12}{p-2}$. Consequently, the selected window $\widehat m_N=N^{\widehat \alpha_N}$ asymptotically growths as $N^{\alpha^*}$ up to a logarithm factor. \\
~\\
Finally, {Proposition} {\ref{hatalpha}} can be used to define an
adaptive estimator of $d$. First, define the straightforward
estimator $\widetilde d_N(N^{\widehat\alpha_N})$,
which should minimize the mean square error using $\widehat \alpha_N$.
However, the estimator $\widetilde d_N(N^{\widehat\alpha_N})$ does not satisfy a
CLT since $\Pr(\widehat \alpha_N\leq
\alpha^*)>0$ and therefore it can not be asserted  that
$\E(\sqrt{N/N^{\widehat\alpha_N}}(\widetilde d_N(N^{\widehat \alpha_N})-d) )= 0$. To
establish a CLT satisfied by an adaptive estimator of $d$, a (few) shifted sequence of $\widehat \alpha_N$, so called $\widetilde
\alpha_N $, has to be considered to ensure $\Pr(\widetilde
\alpha_N \leq \alpha^*)\limiteN 0$. Hence, consider the adaptive scale sequence $(\widetilde
m_N)$ such as  
$$
\widetilde
m_N:=N^{\widetilde \alpha_N}\quad \mbox{with}\quad \widetilde \alpha_N:=\widehat \alpha_N+ \frac {6\, \widehat \alpha_N} {(p-2)(1-\widehat
{\alpha}_N )} \cdot \frac {\log \log N}{\log N}.
$$
and the estimator
$$
\widetilde d_N^{(IR)} :=\widetilde
d_N(\widetilde m_N)=\widetilde d_N(N^{\widetilde \alpha_N}).
$$
The following theorem provides
the asymptotic behavior of the estimator $\widetilde d_N^{(IR)}$:
\begin{thm}\label{tildeD}
Under assumptions of {Proposition} {\ref{hatalpha}},
\begin{eqnarray}\label{CLTD2}
&&\sqrt{\frac{N}{N^{\widetilde \alpha_N }}} \big(\widetilde d_N^{(IR)}  - d \big)
\limiteloiN    {\cal N}\Big (0\, ; \,\Lambda'_0(d)^{-2}\,\big (J_p' \, \Gamma^{-1}_p(d)J_p\big )^{-1}\Big ).
\end{eqnarray}
Moreover, if $\beta\leq 2d+1$,~~~$\displaystyle \forall
\rho>\frac {2(1+3\beta)}{(p-2)\beta},~\mbox{}~~ \frac {N^{\frac
{\beta}{1+2\beta}} }{(\log N)^\rho} \cdot \big|\widetilde d_N^{(IR)} - d \big|
\limiteprobaN  0.$
\end{thm}
\begin{rem}
When $\beta\leq 2d+1$, the adaptive estimator $\widetilde d_N^{(IR)} $ converges to $d$ with a rate of convergence rate equal to the
minimax rate of convergence $N^{\frac{\beta}{1+2\beta}}$ up to a logarithm
factor (this result being classical within this semiparametric
framework). Thus there exists $\ell<0$ such that
$$
N^{\frac{2\beta}{1+2\beta}}(\log N)^{\ell} \E (\widetilde d_N^{(IR)}  -d)^2 <\infty.
$$
Therefore  $\widetilde d_N^{(IR)} $ satisfies an oracle property for the considered semiparametric model.\\
If $\beta>2d+1$, the estimator is not rate optimal. However, simulations (see the following Section) will show that even if $\beta >2d+1$, the rate of convergence of  $\widetilde d_N^{(IR)} $  can be better than the one of the best known rate optimal estimators (local Whittle or global log-periodogram estimators).
\end{rem}
\noindent Moreover an adaptive version of the previous goodness-of-fit test can be derived. Thus define
\begin{equation}\label{TNada}
\widetilde T^{(IR)}_N:=\widehat T_N(N^{\widetilde \alpha_N} ).
\end{equation}
Then,
\begin{prop}\label{testada}
Under assumptions of Proposition \ref{hatalpha},
$$
\widetilde T^{(IR)}_N \limiteloiN \chi^2(p-1).
$$
\end{prop}

\section{Simulations and Monte-Carlo experiments}\label{Simu}
In the sequel, the numerical properties (consistency, robustness, choice of the parameter
$p$) of $\widetilde d_N^{(IR)}$ are investigated.  Then
the simulation results of $\widetilde d_N^{(IR)}$
are compared to those obtained with the best known semiparametric long-memory estimators.
\begin{rem}
Note that all
the softwares (in Matlab language) used in this Section are available with a free access on {\tt
http://samm.univ-paris1.fr/-Jean-Marc-Bardet}.
\end{rem}
\noindent To begin with, the simulation conditions have to be specified.
The results are obtained from $100$ generated independent samples of
each process belonging to the following "benchmark". The concrete
procedures of generation of these processes are obtained from the
circulant matrix method, as detailed in Doukhan {\it et al.} (2003).
The simulations are realized for different values of $d$,
$N$ and processes which satisfy Assumption $S(d,\beta)$:
\begin{enumerate}
\item the fractional Gaussian noise (fGn) of parameter $H=d+1/2$ (for $-0.5<d<0.5$) and $\sigma^2=1$. Such a process is such that
{Assumption} $S(d,2)$ holds;
\item the FARIMA$[p,d,q]$ process with parameter $d$ such that
$d \in (-0.5,0.5)$, the innovation
variance $\sigma^2$ satisfying $\sigma^2=1$ and $p, q\, \in \N$.
A FARIMA$[p,d,q]$  process is such that
Assumption $S(d,2)$ holds;
\item the Gaussian stationary process $X^{(d,\beta)}$, such as its spectral density is
\begin{eqnarray}
f_3(\lambda)=\frac 1 {\lambda^{2d}}(1+\lambda^{\beta})~~~\mbox{for
$\lambda \in [-\pi,0)\cup (0,\pi]$},
\end{eqnarray}
with $d \in (-0.5,0.5)$ and $\beta\in (0,\infty)$. Therefore the spectral density $f_{3}$ is such as
Assumption $S(d,\beta)$ holds.
\end{enumerate}
A "benchmark" which will be considered in the sequel consists of the following particular cases of these processes for $d=-0.4,-0.2,0,0.2,0.4$:
\begin{itemize}
\item fGn processes with parameters $H=d+1/2$;
\item FARIMA$[0,d,0]$ processes with standard Gaussian
innovations;
\item FARIMA$[1,d,1]$ processes with standard Gaussian
innovations and AR coefficient $\phi=-0.3$ and MA coefficient
$\phi=0.7$;
\item $X^{(d,\beta)}$ Gaussian processes with $\beta=1$.
\end{itemize}
\subsection{Application of the IR estimator and tests applied to generated data}
{\bf Choice of the parameter $p$:} This parameter is
important to estimate the "beginning" of the linear part of the
graph drawn by points $(i,IR(im))_i$. On the
one hand, if $p$ is a too small a number (for instance $p=3$),
another small linear part of this graph (even before the "true"
beginning $N^{\alpha^*}$) may be chosen. On the
other hand, if $p$ is a too large a number (for instance
$p=50$ for $N=1000$), the estimator $\widetilde \alpha_N$ will
certainly satisfy $\widetilde \alpha_N<\alpha^*$ since it will not be
possible to consider $p$ different windows larger than
$N^{\alpha^*}$. Moreover, it is
possible that a "good" choice of $p$ depends on the "flatness"
of the spectral density $f$, {\it i.e.} on $\beta$. We have proceeded to simulations for several values of
$p$ (and $N$ and $d$). Only $\sqrt{MSE}$ of estimators are presented.
The results are specified in {Table} {\ref{Table1}}.\\
~\\
\begin{table}[t]
{\scriptsize
\begin{center}
$N=10^3 $ ~\begin{tabular}{|c|c|c|c|c|c|c|}
\hline\hline
Model & Estimates  & $p=5$ & $p=10$ & $p=15$ & $p=20$   \\
\hline \hline
fGn~$(H=d+1/2)$& $\sqrt{MSE}$ $\widetilde d_N^{(IR)}$ &\bf 0.088*& 0.094&0.101 & 0.111 \\
& mean($\widetilde m_N$)  & 11.8&12.5&16.0 &19.4 \\
& $\widehat{ proba} $  & 0.93& 0.89&0.86&0.85\\ \hline
FARIMA$(0,d,0)$& $\sqrt{MSE}$ $\widetilde d_N^{(IR)}$ & 0.112&  0.099&\bf  0.094* &0.107 \\
& mean($\widetilde m_N$)  & 13.9&12.5&14.6 &17.9 \\
& $\widehat{ proba} $  & 0.94& 0.92&0.88&0.86\\ \hline
  FARIMA$(1,d,1)$&  $\sqrt{MSE}$ $\widetilde d_N^{(IR)}$ & 0.141&\bf 0.136*&0.140&0.149 \\
& mean($\widetilde m_N$)  & 15.2&15.0&18.2&21.1 \\
& $\widehat{ proba} $  & 0.94& 0.89&0.86&0.82\\\hline
 $X^{(d,\beta)}$, $\beta=1$&   $\sqrt{MSE}$ $\widetilde d_N^{(IR)}$ & 0.122&\bf 0.112*&0.121&0.123 \\
& mean($\widetilde m_N$)  & 14.1&13.8&16.2&20.0 \\
& $\widehat{ proba} $  & 0.91& 0.90&0.87&0.85\\\hline
\hline
\end{tabular}
\end{center}
\begin{center}
$N=10^4 $ ~\begin{tabular}{|c|c|c|c|c|c|c|}
\hline\hline
Model & Estimates & $p=5$ & $p=10$ & $p=15$ & $p=20$   \\
\hline \hline
fGn~$(H=d+1/2)$& $\sqrt{MSE}$ $\widetilde d_N^{(IR)}$ & 0.030&0.022&0.019 &\bf 0.018* \\
& mean($\widetilde m_N$)  & 13.7&10.3&9.4 &8.9 \\
& $\widehat{ proba} $  & 0.95& 0.89&0.87&0.84\\ \hline
FARIMA$(0,d,0)$& $\sqrt{MSE}$ $\widetilde d_N^{(IR)}$ & 0.039&0.034& 0.033 &\bf  0.031* \\
& mean($\widetilde m_N$)  & 11.5&9.0&8.0 &7.2\\
& $\widehat{ proba} $  & 0.95& 0.90&0.88&0.82\\ \hline
 FARIMA$(1,d,1)$&  $\sqrt{MSE}$ $\widetilde d_N^{(IR)}$ & 0.067&0.062&\bf 0.061*&\bf  0.061* \\
& mean($\widetilde m_N$)  & 18.1&15.9&13.8&13.3\\
& $\widehat{ proba} $  & 0.95& 0.90&0.84&0.78\\\hline
$X^{(d,\beta)}$, $\beta=1$&   $\sqrt{MSE}$ $\widetilde d_N^{(IR)}$ & 0.071&  0.068&\bf  0.067*& 0.071 \\
& mean($\widetilde m_N$)  & 15.2&13.6&11.7&10.9 \\
& $\widehat{ proba} $  & 0.92& 0.88&0.85&0.80\\\hline
\hline
\end{tabular}
\end{center}
\begin{center}
$N=10^5 $ ~\begin{tabular}{|c|c|c|c|c|c|c|}
\hline\hline
Model & Estimates& $p=5$ & $p=10$ & $p=15$ & $p=20$   \\
\hline \hline
fGn~$(H=d+1/2)$& $\sqrt{MSE}$ $\widetilde d_N^{(IR)}$ & 0.012&0.008&0.007 &\bf 0.006* \\
& mean($\widetilde m_N$)  & 14.0&9.8&6.9 &7.9 \\
& $\widehat{ proba} $  & 0.92& 0.90&0.87&0.85\\ \hline
FARIMA$(0,d,0)$& $\sqrt{MSE}$ $\widetilde d_N^{(IR)}$ & 0.021&\bf 0.019*&\bf 0.019* &\bf  0.019* \\
& mean($\widetilde m_N$)  & 15.8&12.7&11.1 &9.8\\
& $\widehat{ proba} $  & 0.96& 0.94&0.92&0.89\\ \hline
 FARIMA$(1,d,1)$&  $\sqrt{MSE}$ $\widetilde d_N^{(IR)}$ & 0.039&0.037&\bf 0.035*&\bf 0.035* \\
& mean($\widetilde m_N$)  & 25.7&21.8&21.4&20.4 \\
& $\widehat{ proba} $  & 0.98& 0.98&0.94&0.93\\\hline
$X^{(d,\beta)}$, $\beta=1$&   $\sqrt{MSE}$ $\widetilde d_N^{(IR)}$ & 0.042&0.042&\bf 0.040*& 0.041 \\
& mean($\widetilde m_N$)  & 22.3&19.9&19.7&16.9 \\
& $\widehat{ proba} $  & 0.99& 0.97& 0.93&0.90\\\hline
\hline
\end{tabular}
\end{center}
}
\caption{\label{Table1} \it  $\sqrt{MSE}$ of the estimator $\widetilde d_N^{(IR)}$, sample mean of the estimator $\widetilde m_N$ and sample frequency that $\widehat{T_N}\leq q_{\chi^2(p-1)}(0.95)$
following $p$ from simulations of the different processes of the benchmark. For each value of $N$ ($10^3$, $10^4$ and $10^5$), of $d$ ($-0.4$, $-0.2$, $0$, $0.2$ and $0.4$) and $p$ ($5$, $10$, $15$, $20$), $100$ independent samples of each process are generated. The values $\sqrt{MSE}$ $\widetilde d_N^{(IR)}$,  mean($\widetilde m_N$) and $\widehat{ proba}$ are obtained from sample mean on the different values of $d$.}
\end{table}
\newline
\noindent {\it Conclusions from {Table} {\ref{Table1}}:} it is clear that $\widetilde d_N^{(IR)}$ converges to $d$ for the four processes, the faster for fGn and FARIMA$(0,d,0)$. The optimal choice of $p$ seems to depend on $N$ for the four processes: $\widehat p=10$ for $N=10^3$, $\widehat p=15$ for $N=10^4$ and $\widehat p\in [15,20]$ for $N=10^5$. The flatness of the spectral density of the process does not seem to have any influence, as well as the value of $d$ (result obtained in the detailed simulations). We will adopt in the sequel the choice $\widehat p=[1.5\, \log(N)]$ reflecting these results. At the contrary to the choice of $m$, this choice of $p$ only depends on $N$ and even if the adaptive scale $\widetilde m_N$ depends on $p$ its value does not change a lot when $p\in \{10,\cdots,20\}$ for $10^3\leq N \leq 10^5$.  \\
Concerning the adaptive choice of $m$, the main point to be remarked is that the smoother the spectral density the smaller $m$; thus $\widetilde m_N$ is smaller for a trajectory of a fGn or a FARIMA$(0,d,0)$ than for a trajectory of a FARIMA$(1,d,1)$ or $X^{(d,1)}$. The choice of $p$ does not appear to significantly affect the value of $\widetilde m_N$.  More detailed results show that the larger $d$ included in $(-0.5,0.5)$ the smaller $\widetilde m_N$: for instance, for the fGn, $N=10^4$ and $p=15$, the mean of $\widetilde m_N$ is respectively equal to $23.9$, $8.3$, $4.5$, $4.2$ and $3.8$ for $d$ respectively equal to $-0.4,~-0.2,~0,~0.2$ and $0.4$. This phenomena can be deduced from the theoretical study since $\alpha^*=(4d+3)^{-1}$ in this case and therefore $\widetilde m_N$ almost growths as $N^{(4d+3)^{-1}}$. \\
Finally, concerning the goodness-of-fit test, we remark that it is too conservative for $p=5$ or $10$ but close to the expected results for $p=15$ and $20$, especially for FARIMA$(1,d,1)$ or $X^{(d,1)}$.\\
~\\
\noindent {\bf Asymptotic distributions of the estimator and test:}
Figure \ref{Figure2} provides the density estimations  of
$\widetilde d_N^{(IR)}$ and $\widetilde T^{(IR)}_N$ for $100$ independent
samples of FGN processes with $d=0.2$ with $N=10^4$ for $p=15$. The goodness-of-fit to the theoretical asymptotic distributions (respectively Gaussian and chi-square) is satisfying.
\begin{figure}[ht]
\[
\epsfxsize 9cm \epsfysize 6cm \epsfbox{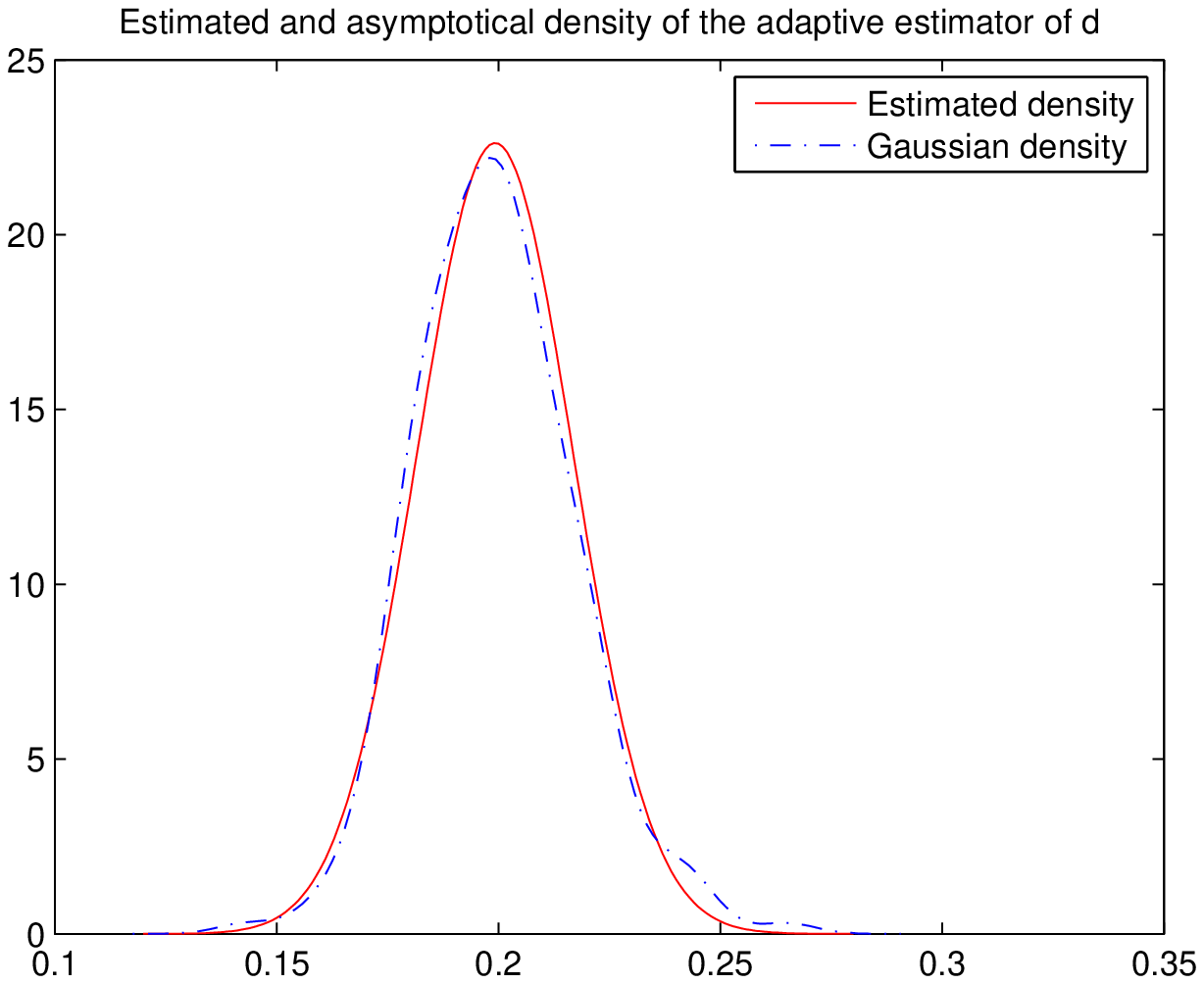}
\hspace*{-0.2 cm} \epsfxsize 9cm \epsfysize 6cm
\epsfbox{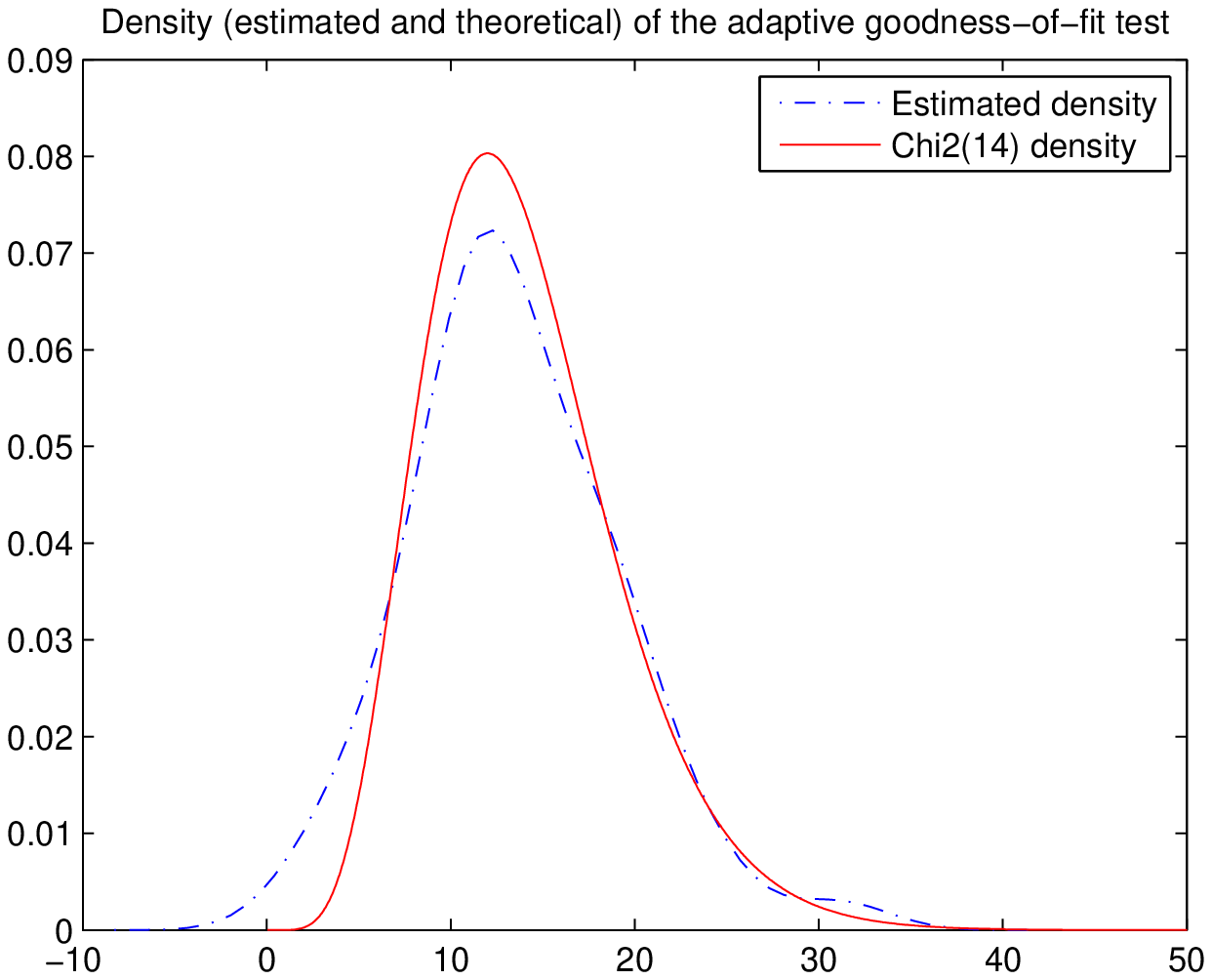}
\]
\caption{\it Density estimations and corresponding theoretical densities of $\widetilde d_N^{(IR)}$ and
$\widetilde T^{(IR)}_N$ for $100$ samples of fGn with $d=0.2$ with $N=10^4$ and $p=15$.}\label{Figure2}
\end{figure}
\\
~\\

\subsection{Comparison with other adaptive semiparametric estimator of the memory parameter}
{\bf Consistency of semiparametric estimators:}
Here we consider the previous "benchmark" and apply the estimator  $\widetilde d_N^{(IR)}$ and $3$ other
 semiparametric estimators of $d$ known for their accuracies are considered:
\begin{itemize}
\item $\widehat d_{MS}$ is the adaptive global log-periodogram estimator introduced
by Moulines and Soulier (1998, 2003), also called FEXP estimator,
with bias-variance balance parameter $\kappa=2$;
\item $\widehat d_{R}$ is the local Whittle estimator introduced by
Robinson (1995). The trimming parameter is $m=N/30$;
\item $\widehat d_{W}$ is an adaptive wavelet based estimator
introduced in Bardet {\it et al.} (2008) using a Lemarie-Meyer type wavelet (another similar choice could be the adaptive wavelet estimator introduced
in Veitch {\it et al.}, 2003, using a Daubechie's wavelet, but its robustness property are quite less interesting).
\item $\widetilde d_N^{(IR)}$ defined previously with $p= [1.5*\log(N)]$.
\item $\widehat d_N(10)$ and $\widehat d_N(30)$ which are the (univariate) IR estimator with $m=10$ and $m=30$ respectively, considered in Surgailis {\it et al.} (2007).
\end{itemize}
Simulation results are reported in {Table} {\ref{Table4}}. \\
\begin{table}[p] {\scriptsize
\begin{center}
$N=10^3~\longrightarrow $ ~\begin{tabular}{|c|c|c|c|c|c|c|}
\hline\hline
Model & $\sqrt{MSE}$ & $d=-0.4$ & $d=-0.2$ & $d=0$ & $d=0.2$ & $d=0.4$ \\
\hline \hline
fGn~$(H=d+1/2)$& $\sqrt{MSE}$ $\widehat d_{MS}$ & 0.102 &0.088&\bf 0.094 *&0.095&0.098 \\
&  $\sqrt{MSE}$ $\widehat d_{R}$ & 0.091 &0.108&0.106 &0.117&0.090\\
&  $\sqrt{MSE}$ $\widehat d_{W}$ & 0.215 &0.103&0.078 &\bf 0.073*& \bf 0.061* \\
&  $\sqrt{MSE}$ $\widetilde d_N^{(IR)}$ &\bf 0.074*  &\bf 0.087*&0.102 & 0.084&0.110 \\
&  $\sqrt{MSE}$ $\widehat d_N(10)$ &0.096  &0.135&0.154 & 0.158&0.154 \\
&  $\sqrt{MSE}$ $\widehat d_N(30)$ &0.112 &0.192&0.246 & 0.270&0.252 \\
\hline  FARIMA$(0,d,0)$& $\sqrt{MSE}$ $\widehat d_{MS}$ & 0.096 &0.096&0.098 &0.096&0.093 \\
&  $\sqrt{MSE}$ $\widehat d_{R}$ & 0.094 &0.113&0.107 &0.112&0.084 \\
&  $\sqrt{MSE}$ $\widehat d_{W}$ & \bf 0.069* &\bf 0.073*&\bf 0.074* &\bf 0.082*&\bf 0.085* \\
&  $\sqrt{MSE}$ $\widetilde d_N^{(IR)}$ & 0.116 &0.085& 0.103 &0.094& 0.101\\
&  $\sqrt{MSE}$ $\widehat d_N(10)$ & 0.139 &0.133& 0.148 & 0.146&0.156\\
&  $\sqrt{MSE}$ $\widehat d_N(30)$ &0.157 & 0.209&0.232 & 0.247&0.243 \\
\hline FARIMA$(1,d,1)$&  $\sqrt{MSE}$ $\widehat d_{MS}$ & 0.098 &\bf 0.092*& \bf 0.089* &\bf 0.088*&0.094 \\
&  $\sqrt{MSE}$ $\widehat d_{R}$ &\bf  0.093*&0.110&0.115 &0.110&\bf  0.089* \\
&  $\sqrt{MSE}$ $\widehat d_{W}$ & 0.108&0.120&0.113 &0.117&0.095 \\
&  $\sqrt{MSE}$ $\widetilde d_N^{(IR)}$ & 0.153 &0.131& 0.135 &0.138&0.123 \\
&  $\sqrt{MSE}$ $\widehat d_N(10)$ & 0.212 &0.188& 0.173 & 0.157&0.155\\
&  $\sqrt{MSE}$ $\widehat d_N(30)$ &0.197 & 0.228&0.250 & 0.265&0.280 \\
\hline  $X^{(D,D')}$, $D'=1$&  $\sqrt{MSE}$ $\widehat d_{MS}$ & 0.092 &\bf 0.089*& \bf 0.113* & \bf 0.107*&\bf 0.100* \\
&  $\sqrt{MSE}$ $\widehat d_{R}$ & 0.093 &0.111&0.129 &0.124&0.111 \\
&  $\sqrt{MSE}$ $\widehat d_{W}$ & 0.217 &0.209&0.211 &0.201&0.189 \\
&  $\sqrt{MSE}$ $\widetilde d_N^{(IR)}$ &\bf 0.075* &0.101& 0.121 &0.122&0.131 \\
&  $\sqrt{MSE}$ $\widehat d_N(10)$ &0.109  &0.143&0.163 & 0.168&0.180 \\
&  $\sqrt{MSE}$ $\widehat d_N(30)$ &0.109  &0.177&0.228 & 0.249&0.247 \\
\hline
\end{tabular}
\end{center}}
{\scriptsize
\begin{center}
$N=10^4~\longrightarrow $ ~\begin{tabular}{|c|c|c|c|c|c|c|}
\hline\hline
Model & $\sqrt{MSE}$& $d=-0.4$ & $d=-0.2$ & $d=0$ & $d=0.2$ & $d=0.4$ \\
\hline \hline
fGn~$(H=d+1/2)$& $\sqrt{MSE}$ $\widehat d_{MS}$ & 0.040 &0.031&0.032 &0.035&0.035 \\
&  $\sqrt{MSE}$ $\widehat d_{R}$ & 0.040 &0.027&0.029 &0.031&0.030 \\
&  $\sqrt{MSE}$ $\widehat d_{W}$ & 0.129 &0.045&0.026 &0.022& 0.020 \\
&  $\sqrt{MSE}$ $\widetilde d_N^{(IR)}$ &\bf 0.019*  &\bf 0.019*&\bf 0.017* & \bf 0.016*&\bf 0.019* \\
&  $\sqrt{MSE}$ $\widehat d_N(10)$ & 0.036  &0.038&0.049 & 0.043&0.048 \\
&  $\sqrt{MSE}$ $\widehat d_N(30)$ &0.043  &0.070& 0.086 & 0.081&0.076 \\
\hline  FARIMA$(0,d,0)$& $\sqrt{MSE}$ $\widehat d_{MS}$ & 0.036 &0.030&0.031 &0.035&0.032 \\
&  $\sqrt{MSE}$ $\widehat d_{R}$ & 0.031 &0.028&0.027 &0.029&0.029 \\
&  $\sqrt{MSE}$ $\widehat d_{W}$ & \bf 0.020*&\bf 0.018*& 0.023 &0.025&\bf 0.028* \\
&  $\sqrt{MSE}$ $\widetilde d_N^{(IR)}$ & 0.066 &0.031& \bf 0.018* &\bf 0.020*&\bf  0.028*\\
&  $\sqrt{MSE}$ $\widehat d_N(10)$ & 0.076&0.047&0.043 & 0.053&0.038 \\
&  $\sqrt{MSE}$ $\widehat d_N(30)$ &0.074  &0.085&0.073 & 0.086&0.073 \\
\hline FARIMA$(1,d,1)$&  $\sqrt{MSE}$ $\widehat d_{MS}$ & 0.035 & 0.033&  0.032 & 0.036&0.031 \\
&  $\sqrt{MSE}$ $\widehat d_{R}$ &\bf  0.031* & \bf 0.029*& \bf  0.030* & \bf 0.032*&\bf  0.027* \\
&  $\sqrt{MSE}$ $\widehat d_{W}$ & 0.054&0.054&0.050 &0.052&0.048 \\
&  $\sqrt{MSE}$ $\widetilde d_N^{(IR)}$ & 0.099 &0.066& 0.052 &0.047&0.046 \\
&  $\sqrt{MSE}$ $\widehat d_N(10)$ & 0.141  &0.095& 0.075 & 0.055&0.051 \\
&  $\sqrt{MSE}$ $\widehat d_N(30)$ &0.111  &0.085&0.094 & 0.090&0.074 \\
\hline  $X^{(D,D')}$, $D'=1$&  $\sqrt{MSE}$ $\widehat d_{MS}$ & 0.029 & \bf 0.037*&\bf 0.035* &\bf 0.041*&\bf 0.038* \\
&  $\sqrt{MSE}$ $\widehat d_{R}$ & 0.032 &0.041&0.037 &\bf 0.041*&0.039 \\
&  $\sqrt{MSE}$ $\widehat d_{W}$ & 0.110 &0.115&0.115 &0.112&0.114 \\
&  $\sqrt{MSE}$ $\widetilde d_N^{(IR)}$ &\bf 0.018* & 0.064& 0.092 &0.084&0.081 \\
&  $\sqrt{MSE}$ $\widehat d_N(10)$ &0.035  &0.093& 0.102 & 0.106&0.094 \\
&  $\sqrt{MSE}$ $\widehat d_N(30)$ & 0.039  &0.088&0.084 & 0.074&0.077 \\
\hline
\end{tabular}
\end{center}}
\caption {\label{Table4} \it Comparison of the different log-memory parameter estimators for processes of the benchmark.
For each process and value of $d$ and $N$, $\sqrt{MSE}$ are computed from $100$ independent generated samples. }
\end{table}
\newline
\noindent {\it Conclusions from Table \ref{Table4}:} The adaptive IR estimator $\widetilde d_N^{(IR)}$ numerically shows a convincing convergence rate with respect to the other estimators. \\
The estimators $\widehat d_N(10)$ and $\widehat d_N(30)$ are clearly the worst estimators of $d$. This can be explained by two facts:
1/ the numerical expression of the matrix $\widehat \Sigma_N(m)$ is almost a diagonal matrix: therefore a least square regression using several window lengths provides better estimations than an estimator using only one window length; 2/ $\widehat d_N(10)$ and $\widehat d_N(30)$ use a fixed window length ($m=10$ and $m=30$) for any process and $N$ while we know that $m\simeq N^{\alpha^*}$ is the optimal choice which is approximated by $\widetilde m_N$. \\
Both the ``spectral'' estimator $\widehat d_R$ and $\widehat d_{MS}$ provide more stable results that do not depend very much on $d$ and the process, while the wavelet based estimator $\widehat d_{W}$ and $\widetilde d_N^{(IR)}$ are more sensible to the flatness of the spectral density. But, especially for ``smooth processes'' (fGn and FARIMA$(0,d,0)$), $\widetilde d_N^{(IR)}$ is a very accurate semiparametric estimator and is globally more efficient than the other estimators.\\
\newline
{\bf Robustness of the different semiparametric estimators:} To
conclude with the  numerical properties of the estimators, five different
processes not satisfying Assumption $S(d,\beta)$ are considered:
\begin{itemize}
\item a FARIMA$(0,d,0)$ process with innovations satisfying a uniform law;
\item a FARIMA$(0,d,0)$ process with innovations satisfying a symmetric Burr distribution with cumulative distribution function $F(x)=1-\frac 1 2 \frac 1 {1+x^2}$ for $x \geq 0$ and $F(x)=\frac 1 2 \frac 1 {1+x^2}$ for $x \leq 0$ (and therefore
$\E |X_i|^2=\infty$ but $\E |X_i|<\infty$);
\item a FARIMA$(0,d,0)$ process with innovations satisfying a symmetric Burr distribution with cumulative distribution function $F(x)=1-\frac 1 2  \frac 1 {1+|x|^{3/2}}$ for $x \geq 0$ and $F(x)=\frac 1 2 \frac 1 {1+|x|^{3/2}}$ for $x \leq 0$ (and therefore
$\E |X_i|^2=\infty$ but $\E |X_i|<\infty$);
\item a Gaussian stationary process  with a spectral density $f(\lambda)=||\lambda|-\pi/2|^{-2d}$
for all $\lambda\in [-\pi,\pi] \setminus \{-\pi/2,\pi/2\}$: this is a GARMA$(0,d,0)$ process. The
local behavior of $f$ in $0$ is $f(|\lambda|) \sim (\pi/2)^{-2d}\,
|\lambda|^{-2d}$ with $d=0$, but the smoothness condition for $f$ in
Assumption $S(0,\beta)$ is not satisfied.
\item a trended fGn with parameter $H=d+0.5$ and an additive linear trend;
\item a fGn ($H=d+0.5$) with an additive linear trend	and an additive sinusoidal seasonal component of period $T=12$.
\end{itemize}
The results of these simulations are given in {Table} {\ref{Table5}}.\\
\begin{table}[p] {\scriptsize
\begin{center}
$N=10^3~\longrightarrow $ ~\begin{tabular}{|c|c|c|c|c|c|c|}
\hline\hline
Model+Innovation & $\sqrt{MSE}$& $d=-0.4$ & $d=-0.2$ & $d=0$ & $d=0.2$ & $d=0.4$ \\
\hline \hline
FARIMA$(0,d,0)$ Uniform & $\sqrt{MSE}$ $\widehat d_{MS}$ & 0.189 &0.090&0.091 &\bf 0.082*&0.092 \\
&  $\sqrt{MSE}$ $\widehat d_{R}$ & 0.171 &0.104&0.109 &0.102&\bf 0.086*\\
&  $\sqrt{MSE}$ $\widehat d_{W}$ &\bf 0.111* &\bf 0.066*&\bf 0.072* &0.118& 0.129 \\
&  $\sqrt{MSE}$ $\widetilde d_N^{(IR)}$ & 0.186  & 0.081& 0.083 & 0.112&0.093 \\
\hline FARIMA$(0,d,0)$ Burr ($\alpha=2$) & $\sqrt{MSE}$ $\widehat d_{MS}$ & 0.174 &0.087&0.092 &0.084&\bf 0.091* \\
&  $\sqrt{MSE}$ $\widehat d_{R}$ & 0.183 &0.104&0.097 &0.107&0.079 \\
&  $\sqrt{MSE}$ $\widehat d_{W}$ & \bf 0.149* &\bf 0.086*& 0.130 &0.101&0.129 \\
&  $\sqrt{MSE}$ $\widetilde d_N^{(IR)}$ & 0.221 &0.119&\bf 0.076* &\bf 0.082*& 0.139\\
\hline FARIMA$(0,d,0)$ Burr ($\alpha=3/2$) & $\sqrt{MSE}$ $\widehat d_{MS}$ &  0.188 &\bf 0.087*&\bf 0.063* &\bf 0.099*& 0.075 \\
&  $\sqrt{MSE}$ $\widehat d_{R}$ & \bf 0.183* &0.110&0.079 &0.125&\bf 0.072* \\
&  $\sqrt{MSE}$ $\widehat d_{W}$ &  0.219 & 0.108& 0.138 &0.146&0.159 \\
&  $\sqrt{MSE}$ $\widetilde d_N^{(IR)}$ & 0.264 &0.134&0.094 & 0.155& 0.187\\
\hline  GARMA$(0,d,0)$ &  $\sqrt{MSE}$ $\widehat d_{MS}$ & 0.149 &0.109&  0.086 &  0.130& 0.172 \\
&  $\sqrt{MSE}$ $\widehat d_{R}$ & \bf 0.098* &0.104&0.090 &0.132&\bf 0.125* \\
&  $\sqrt{MSE}$ $\widehat d_{W}$ & 0.117 &\bf 0.074*&\bf 0.081* &0.182&0.314 \\
&  $\sqrt{MSE}$ $\widetilde d_N^{(IR)}$ &0.124 &0.121& 0.110 &\bf 0.102*&0.331 \\
\hline  Trend &  $\sqrt{MSE}$ $\widehat d_{MS}$ & 1.307 &0.891&  0.538 &  0.290& 0.150 \\
&  $\sqrt{MSE}$ $\widehat d_{R}$ & 0.900 &0.700&0.498 &0.275&0.087 \\
&  $\sqrt{MSE}$ $\widehat d_{W}$ & \bf 0.222* & \bf 0.103*&0.083 &  0.071&\bf 0.059* \\
&  $\sqrt{MSE}$ $\widetilde d_N^{(IR)}$ & 1.65 & 0.223&\bf  0.079* &\bf 0.050*&  0.076 \\
\hline Trend + Seasonality &  $\sqrt{MSE}$ $\widehat d_{MS}$ & 1.178 &0.803& 0.477 & 0.238& 0.123 \\
&  $\sqrt{MSE}$ $\widehat d_{R}$ & 0.900 &0.700&0.498 &0.284&\bf 0.091* \\
&  $\sqrt{MSE}$ $\widehat d_{W}$ & \bf 0.628* &\bf 0.407*&0.318 &0.274&0.283 \\
&  $\sqrt{MSE}$ $\widetilde d_N^{(IR)}$ &1.54 &1.01& \bf 0.311* &\bf 0.158*&0.145 \\
\hline
\end{tabular}
\end{center}}
{\scriptsize
\begin{center}
$N=10^4~\longrightarrow $ ~\begin{tabular}{|c|c|c|c|c|c|c|}
\hline\hline
Model+Innovation & $\sqrt{MSE}$& $d=-0.4$ & $d=-0.2$ & $d=0$ & $d=0.2$ & $d=0.4$ \\
\hline \hline
FARIMA$(0,d,0)$ Uniform & $\sqrt{MSE}$ $\widehat d_{MS}$ & 0.177 &0.039&0.033 &0.034&0.034 \\
&  $\sqrt{MSE}$ $\widehat d_{R}$ & 0.171 &0.032&0.030 &0.028& \bf 0.032*\\
&  $\sqrt{MSE}$ $\widehat d_{W}$ &\bf  0.125* &\bf 0.027*&0.025 & 0.028&  0.035 \\
&  $\sqrt{MSE}$ $\widetilde d_N^{(IR)}$ & 0.165  & 0.042&\bf 0.017* &\bf  0.027*&\bf 0.032* \\
\hline FARIMA$(0,d,0)$ Burr ($\alpha=2$) & $\sqrt{MSE}$ $\widehat d_{MS}$ & 0.180 &0.036&0.041 &0.033&0.032 \\
&  $\sqrt{MSE}$ $\widehat d_{R}$ & 0.169 &\bf 0.031*&0.030 &\bf 0.031*&\bf 0.029* \\
&  $\sqrt{MSE}$ $\widehat d_{W}$ & \bf 0.138* &0.068& 0.065 &0.076&0.066 \\
&  $\sqrt{MSE}$ $\widetilde d_N^{(IR)}$ & 0.219 &0.067& \bf 0.018* & 0.039& 0.074\\
\hline FARIMA$(0,d,0)$ Burr ($\alpha=3/2$) & $\sqrt{MSE}$ $\widehat d_{MS}$ & 0.18 &0.038&\bf 0.026* &0.030&\bf 0.021* \\
&  $\sqrt{MSE}$ $\widehat d_{R}$ & 0.174&\bf 0.033*&0.031 &\bf 0.023*&0.023 \\
&  $\sqrt{MSE}$ $\widehat d_{W}$ & \bf 0.126* &0.058& 0.149 &0.124&0.090 \\
&  $\sqrt{MSE}$ $\widetilde d_N^{(IR)}$ & 0.264 &0.113& 0.030 & 0.099& 0.159\\
\hline  GARMA$(0,d,0)$ &  $\sqrt{MSE}$ $\widehat d_{MS}$ & 0.063 &0.041& 0.028 &  0.032&0.060 \\
&  $\sqrt{MSE}$ $\widehat d_{R}$ &\bf  0.037* &\bf 0.033*&0.025 &\bf 0.026* &\bf 0.030* \\
&  $\sqrt{MSE}$ $\widehat d_{W}$ & 0.061 &0.052&0.021 &0.078&0.081 \\
&  $\sqrt{MSE}$ $\widetilde d_N^{(IR)}$ & 0.074 & 0.040& \bf 0.016* &0.055&0.109 \\
\hline  Trend &  $\sqrt{MSE}$ $\widehat d_{MS}$ & 1.16 &0.785& 0.450 & 0.171& 0.072 \\
&  $\sqrt{MSE}$ $\widehat d_{R}$ & 0.900 &0.700&0.431 &0.192&0.067 \\
&  $\sqrt{MSE}$ $\widehat d_{W}$ & 0.135 &0.046&\bf 0.021* &0.019&0.021 \\
&  $\sqrt{MSE}$ $\widetilde d_N^{(IR)}$ &\bf 0.019* &\bf 0.021*& \bf 0.021* &\bf 0.016*&\bf 0.020* \\
\hline Trend + Seasonality &  $\sqrt{MSE}$ $\widehat d_{MS}$ & 1.219 & 0.841& 0.474 & 0.194& 0.099 \\
&  $\sqrt{MSE}$ $\widehat d_{R}$ & 0.900 &0.700&0.431 &0.189&0.063 \\
&  $\sqrt{MSE}$ $\widehat d_{W}$ & \bf 0.097* &\bf 0.073*&0.063 &0.065& 0.051 \\
&  $\sqrt{MSE}$ $\widetilde d_N^{(IR)}$ &0.671 &0.382&\bf  0.049* &\bf 0.047*&\bf 0.041* \\
\hline
\end{tabular}
\end{center}}
\caption {\label{Table5} \it Comparison of the different log-memory parameter estimators for processes of the benchmark.
For each process and value of $d$ and $N$, $\sqrt{MSE}$ are computed from $100$ independent generated samples. }
\end{table}
\newline
\noindent ~\\
{\it Conclusions from {Table} {\ref{Table5}}:}  The main advantage of $\widehat d_{W}$ and $\widetilde d_N^{(IR)}$ with respect to $\widehat d_{MS}$ and $\widehat d_{R}$ is exhibited in this table: they are robust with respect to smooth trends, especially in the case of long memory processes ($d>0$). This has already been observed in Bruzaite and Vaiciulis (2008) for IR statistic (and even for certain discontinuous trends). Both those estimators are also robust with respect to seasonal component and this robustness would have been improved if we had chosen $m$ (or scales) as a multiple of the period (which is generally known). \\
The second good surprise of these simulations is that the adaptive IR estimator $\widetilde d_N^{(IR)}$ is also consistent for non Gaussian distributions even if the function $\Lambda$ in (\ref{Lambdad}) and therefore all our results are typically obtained for Gaussian distributions. The case of finite-variance processes is not surprising (see Remark 2). But this is more surprising for infinite variance processes. A first explanation of this was given in Surgailis {\it et al.} (2007) in the case of i.i.d.r.v. in the
domain of attraction of a stable law with index $0<\alpha<2 $ and skewness parameter $-1\leq \beta \leq  1$: they concluded that $IR_N(m)$ converges to almost the same limit. The extension to $\alpha$-stable linear processes of this first explanation should require technical developments but the expression of the IR statistic (which is bounded in $[0,1]$ for any processes) could allow to apply it to infinite variance processes.   Note that the other semiparametric estimators  are also consistent in such frame with faster convergence rates notably for the local Whittle estimator.

\section{Proofs}\label{proofs}
\begin{proof}[Proof of Property \ref{MCLT}]
We proceed in two steps. \\
{\bf Step 1:} First, we compute the limit of $\frac{N}{m}\,\Cov \big (IR_N(jm), IR_N(j'm) \big)$ when $N,~m$ and $N/m \to \infty$. As in Surgailis {\it et al} (2007), define also for all $j=1,\cdots,p$ and $k=1,\cdots, N-3m_j$ (with $m_j=jm$):
\begin{eqnarray}\label{Yjk}
Y_{m_j}(k)&:=&\frac{1}{V_{m_j}}\sum_{t=k+1}^{k+m_{j}}(X_{t+m_{j}}-X_{t})~~,~\mbox{with}~~ V_{m_j}^{2}:=\E \Big [ \Big (\sum_{t=k+1}^{k+m_{j}}(X_{t+m_{j}}-X_{t})\Big )^2\Big ]\\
\mbox{and}\quad \eta_{m_j}(k)&:=&\frac{|Y_{m_j}(k)+Y_{m_j}(k+m_{j})|}{|Y_{m_j}(k)|+|Y_{m_j}(k+m_{j})|}.
\end{eqnarray}
Note that $Y_{m_j}(k) \sim {\cal N}(0,1)$ for any $k$ and $j$ and
$$
IR_N(m_j)=\frac{1}{N-3m_{j}}\sum_{k=0}^{N-3m_{j}-1} \eta_{m_j}(k)\quad \mbox{for 	all}~j=1,\cdots p.
$$

\begin{eqnarray*}
\Cov(IR_N(m_j),IR_N(m_{j'}))&=&\frac{1}{N-3m_{j}}~\frac{1}{N-3m_{j'}} \, \sum_{k=0}^{N-3m_{j}-1}\sum_{k'=0}^{N-3m_{j'}-1}\Cov(\eta_{m_j}(k),\eta_{m_{j'}}(k')))\\
&=&\frac{1}{(\frac N {m_{j}}-3)(\frac N {m_{j'}}-3)}\int_{\tau=0}^{\frac {N-1} {m_{j}}-3}\int_{\tau'=0}^{\frac {N-1} {m_{j'}}-3}\Cov(\eta_{m_j}([m_{j}\tau]),\eta_{m_{j'}}([m_{j'}\tau'])))\,d\tau \,d\tau'.
\end{eqnarray*}
Now according to (5.20) of the same article, with $\longrightarrow_{FDD}$ denoting the finite distribution convergence when $m\to \infty$,
\begin{eqnarray*}
Y_m([m\tau])~\longrightarrow_{FDD}~Z_{d}( \tau)
\end{eqnarray*}
where $Z_d$ is defined in (\ref{Zd}). Now
\begin{eqnarray*}
Y_{jm}(k)&=&\frac{1}{V_{m_j}}  \sum_{t=1}^{jm}X_{t+jm+1}-\sum_{t=1}^{jm}X_{t+1}X_{t}) \\
& =& \frac{1}{V_{m_j}} \sum_{i=-(j-1)}^{j-1} (j-|i|) V_m \, Y_m(t+(j+i-1)m). 
\end{eqnarray*}
But $V^2_m \sim c_0 V(d) m^{2d+1}$ when $m\to \infty$ (see (2.20) in Surgailis {\it et al}, 2007). Therefore we obtain
$Y_{jm}([mj\tau]) \sim \frac 1 {j^{d+1/2}} \, \sum_{i=-(j-1)}^{j-1} (j-|i|) Y_m([mj\tau]+(j+i-1)m)$ when $m\to \infty$ (in distribution) and more generally,
\begin{multline}\label{newcov}
\hspace{-0.3cm}\big (Y_{jm}([mj\tau]) , Y_{j'm}([mj'\tau']) ~\longrightarrow_{FDD}~\\ \Big ( \frac 1 {j^{d+1/2}} \sum_{i=-(j-1)}^{j-1} (j-|i|) Z_d(j\tau+j+i-1)\, , \, \frac 1 {(j')^{d+1/2}}  \sum_{i'=-(j'-1)}^{j'-1} (j'-|i'|) Z_d(j'\tau'+j'+i'-1)\Big ),
\end{multline}
when $m\to \infty$. Hence, obvious computations lead to define for $t\in \R$ 
\begin{eqnarray}\label{Zdj}
Z^{(j)}_d(t)&:= & \hspace{-0.5cm}\sum_{i=-(j-1)}^{j-1} (j-|i|) Z_d(t+j+i-1)=\frac{B_{d+0.5}(t+2j)-2B_{d+0.5}(t+j)+B_{d+0.5}(t)}
{\sqrt{|4^{d+0.5}-4|}} \\
\label{gammaj} \gamma^{(j,j')}_d(t)&:=&\Cov\big (\psi(Z_{d}^{(j)}(0),Z_{d}^{(j)}(j)),\psi(Z_{d}^{(j')}(t),Z_{d}^{(j')}(t+j' ))\big ).
\end{eqnarray}
Now, as the function $\psi(x,y)=\frac{|x+y|}{|x|+|y|}$ is a continuous (on $\R^2\setminus \{0,0\}$) and bounded function (with $0\leq\psi(x,y)\leq 1$)
and since $\eta_{m_j}([mj\tau])=\psi(Y_{m_j}([m_j\tau]),Y_{m_j}([m_j(\tau+1)]))$,
then from (\ref{newcov}),
\begin{eqnarray*}
\Cov\big (\eta_{m_j}([m_{j}\tau]),\eta_{m_{j'}}([m_{j'}\tau'])\big )&\limitem &\Cov\big  ( \psi ( Z^{(j)}_d(j\tau),Z^{(j)}_d(j(\tau+1))),\psi(Z^{(j')}_d(j'\tau'),Z^{(j')}_d(j'(\tau'+1)))\big )\\
&\limitem & \gamma^{(j,j')}_d(j'\tau'-j\tau),
\end{eqnarray*}
using the stationarity of the process $Z_d$ and therefore of processes $Z_d^{(j)}$ and $Z_d^{(j')}$. Hence, when $N,\, m$ and $N/m\to \infty$,
\begin{eqnarray}\label{equacov}
\nonumber  \frac N m \, \Cov(IR_N(j  m),IR_N(j' m)) &   \sim &  \frac{N}{m(\frac N{jm}-3)(\frac N {j'm}-3)}  \\
\nonumber  &  &  \hspace{-0.2cm} \times  \int_{0}^{\frac {N-1}{jm}-3} \hspace{-0.2cm}  \int_{0}^{\frac {N-1} {j'm}-3} \hspace{-0.6cm}   \Cov\big (\psi(Z_{d}^{(j)}(j \, \tau),Z_{d}^{(j)}(j \, \tau+j)),\psi(Z_{d}^{(j')}(j' \, \tau'),Z_{d}^{(j')}(j' \, \tau'+j' ))\big )d\tau d\tau' \\
\nonumber &\sim & \frac{mN}{(N-3jm)(N -3j'm)}\, \int_{0}^{\frac {N-1}{m}-3j}\int_{0}^{\frac {N-1}{m}-3j'}\gamma^{(j,j')}_d(s'-s)\,ds\, ds' \\
\nonumber & \sim & \frac{m}{N}\, \int_{-\frac {N}{m}}^{\frac {N}{m}}\big (\frac {N}{m}-|u| \big )\,  \gamma^{(j,j')}_d(u)\,du \\
&\longrightarrow & \int_{-\infty}^{\infty}\gamma^{(j,j')}_d(u)\,du=:\sigma_{j,j'}(d).
\end{eqnarray}
This last limit is obtained, {\em mutatis mutandis}, from the relation (5.23) Surgailis {\it et al} (2007), and thus
$\gamma^{(j,j')}_d(u)=C\, (u^{-2}\wedge 1)$, implying $ \frac{m}{N}\, \int_{-\frac {N}{m}}^{\frac {N}{m}} |u| \,  \gamma^{(j,j')}_d(u)\,du \limiteNm  0$.
It achieves the first step of the proof.\\
~\\
{\bf Step 2:} It remains to prove the multidimensional central limit theorem. Then consider a linear combination of $(IR_N(m_j))_{1 \leq j \leq p}$, {\it i.e.} $\sum_{j=1}^p u_j \, IR_N(m_j)$ with $(u_1,\cdots,u_p) \in \R^p$. For ease of notation, we will restrict our purpose to $p=2$, with $m_i=r_i m$ where $r_1\leq r_2$ are fixed positive integers. Then with the previous notations and following the notations and results of Theorem 2.5 of Surgailis {\it et al.} (2007):
\begin{multline*}
u_1 \, IR_N(r_1m)+u_2 \, IR_N(r_2m)= u_1 (\E [IR_N(r_1m)] +S_K(r_1m) +\widetilde S_K(r_1m))\\
+u_2 (\E [IR_N(r_2m)] +S_K(r_2m) +\widetilde S_K(r_2m)).
\end{multline*}
From (5.31) of Surgailis {\it et al.} (2007), we have $\widetilde S_K(m_1)=o(S_K(m_1))$ and $\widetilde S_K(m_2) = o(S_K(m_2))$ when $K\to \infty$ and from an Hermitian decomposition  $(N/m)^{1/2}(u_1 S_K(m_i)+u_2 S_K(m_2)) \to_D {\cal N}(0,\gamma_K^2)$ as $N$, $m$ and $N/m \to \infty$ since the cumulants of $(N/m)^{1/2}(u_1 S_K(m_i)+u_2 S_K(m_2))$ of order greater or equal to $3$ converge to $0$ (since this result is proved for each $S_K(m_i)$). Moreover, from the previous computations, $\gamma_K^2 \to (u_1^2\sigma_{ r_1, r_1}(d) +2 u_1u_2 \sigma_{ r_1, r_2}(d) + u_2^2 \sigma_{ r_2, r_2}(d))$ when $K \to \infty$. Therefore the multidimensional central limit theorem is established.
\end{proof}
\begin{property}\label{devEIR}
Let $X$ satisfy Assumption $S(d,\beta)$ with $-0.5<d<0.5$ and $\beta>0$. Then, there exists a constant $K(d,\beta)<0$ depending only on $d$ and $\beta$ such as\\
\begin{tabular}{lcll}
$\E \big [IR_N(m)\big ]$&$=$&
$\Lambda_0(d)+K(d,\beta)\times m^{-\beta}+O\big(m^{-\beta-\varepsilon}+m^{-2d-1}\log(m)\big)$ & if $-2d+\beta<1$, \\
&$ = $&$ \Lambda_0(d)+ K(d,\beta) \times m^{-\beta}\, \log(m)~~+O\big(m^{-\beta}\big)$
& if  $-2d+\beta=1$;\\
&$=$&$\Lambda_0(d)+O\big(m^{-2d-1}\big)$&if  $-2d+\beta>1$.\\
\end{tabular}
\end{property}
\begin{proof}[Proof of Property \ref{devEIR}]
As in Surgailis {\it et al} (2007), we can write:
\begin{eqnarray} \label{RmVm}
\nonumber \E \big [ IR_N(m) \big ]=\E\big(\frac{|Y^{0}+Y^{1}|}{|Y^{0}|+|Y^{1}|}\big)=\Lambda(\frac{R_{m}}{V_{m}^{2}}) \quad
\mbox{with}\quad
\frac{R_{m}}{V_{m}^{2}}:=1-2 \, \frac{\int_{0}^{\pi}f(x)\frac{\sin^{6}(\frac{mx}{2})}{\sin^{2}(\frac{x}{2})}dx~}{\int_{0}^{\pi}f(x)
\frac{\sin^{4}(\frac{mx}{2})}{\sin^{2}(\frac{x}{2})}dx}.
\end{eqnarray}
Therefore an expansion of $R_{m}/V_{m}^{2}$ will provide an expansion of $\E \big [ IR_N(m) \big ]$ when $m\to \infty$ and the multidimensional CLT (\ref{cltd}) will be deduced from Delta-method. \\
~\\
{\bf Step 1} Let $f$ satisfy Assumption $S(d,\beta)$. Then we are going to establish that there exist positive real numbers $C_{1}$ and $C_{2}$ specified in (\ref{C0C1}) and (\ref{C01C02}) and such that:
\begin{eqnarray*}
1.& \mbox{if $-1<-2d<1$ and $-2d+\beta<1$,} & \frac{R_{m}}{V_{m}^{2}}=\rho(d)~+~C_{1}(-2d,\beta)~~m^{-\beta}+O\big(m^{-\beta-\varepsilon}+m^{-2d-1}\log m\big) \\
2.& \mbox{if $-1<-2d<1$ and $-2d+\beta=1$,} & \frac{R_{m}}{V_{m}^{2}}=\rho(d)+C_{2}(1-\beta,\beta) ~m^{-\beta}\log m+O\big(m^{-\beta}\big) \\
3.&\mbox{if $-1<-2d<1$ and $-2d+\beta>1$,} & \frac{R_{m}}{V_{m}^{2}}=\rho(d)+O\big(m^{-2d-1}\big).
\end{eqnarray*}
Indeed under Assumption $S(d,\beta)$ and with $J_j(a,m),\, j=4,6,$ defined in (\ref{Jj}) of Lemma \ref{lemma46} (see below), it is clear that,
$$
\frac{R_{m}}{V_{m}^{2}}=1-2\, \frac{J_6(-2d,m)+ \frac{c_{1}}{c_{0}}J_6(-2d+\beta,m) +O(J_6(-2d+\beta+\varepsilon))}
{J_4(-2d,m)+ \frac{c_{1}}{c_{0}}J_4(-2d+\beta,m) +O(J_4(-2d+\beta+\varepsilon))},
$$
since $\displaystyle \int_{0}^{\pi}O(x^{-2d+\beta+\varepsilon})\frac{\sin^{j}(\frac{mx}{2})}{\sin^{2}(\frac{x}{2})}dx=O(J_j(-2d+\beta+\varepsilon))$ for $j=4,6$. Now we follow the results of Lemma \ref{lemma46}:\\
~\\
1. Let $-1<-2d+\beta<1$. Then for any $\varepsilon>0$,
\begin{eqnarray*}
\frac{R_{m}}{V_{m}^{2}}& \hspace{-3mm}=&\hspace{-3mm} 1\hspace{-1mm} -\hspace{-1mm}2\, \frac{
C_{61}(-2d)m^{1+2d}\hspace{-1mm}+\hspace{-1mm}C_{62}(-2d)\hspace{-1mm}+\hspace{-1mm}\frac{c_{1}}{c_{0}}\big (C_{61}(-2d+\beta)m^{1+2d-\beta}\hspace{-1mm}+\hspace{-1mm}C_{62}(-2d+\beta)\big )\hspace{-1mm}+\hspace{-1mm}O\big(m^{1+2d-\beta-\varepsilon}\hspace{-1mm}+\hspace{-1mm}\log m \big) }
{C_{41}(-2d)m^{1+2d}\hspace{-1mm}+\hspace{-1mm}C_{42}(-2d)\hspace{-1mm}+\hspace{-1mm}\frac{c_{1}}{c_{0}}\big (C_{41}(-2d+\beta)m^{1+2d-\beta}\hspace{-1mm}+\hspace{-1mm}C_{42}(-2d+\beta)\big )\hspace{-1mm}+\hspace{-1mm}O\big(m^{1+2d-\beta-\varepsilon}\hspace{-1mm}+\hspace{-1mm}\log m \big)}\\
&\hspace{-3mm}=&\hspace{-3mm}1\hspace{-1mm}-\hspace{-1mm}\frac{2}{C_{41}(-2d)}\Big[C_{61}(-2d)\hspace{-1mm}+\hspace{-1mm}\frac{c_{1}}{c_{0}}C_{61}(-2d+\beta)m^{-\beta}\Big]\Big[1\hspace{-1mm}-\hspace{-1mm}\frac{c_{1}}{c_{0}}\frac{C_{41}(-2d+\beta)}{C_{41}(-2d)}m^{-\beta}\Big]\hspace{-1mm}+\hspace{-1mm}O\big(m^{-\beta-\varepsilon}\hspace{-1mm}+\hspace{-1mm}m^{-2d-1}\log m \big)\\
&\hspace{-3mm}=&\hspace{-3mm}1\hspace{-1mm}-\hspace{-1mm}\frac{2C_{61}(-2d)}{C_{41}(-2d)}\hspace{-1mm}+\hspace{-1mm}2\frac{c_{1}}{c_{0}}\Big[\frac{C_{61}(-2d)C_{41}(-2d+\beta)}{C_{41}(-2d)C_{41}(-2d)}\hspace{-1mm}-\hspace{-1mm}\frac{C_{61}(-2d+\beta)}{C_{41}(-2d)}\Big]m^{-\beta}\hspace{-1mm}+\hspace{-1mm}O\big(m^{-\beta-\varepsilon}\hspace{-1mm}+\hspace{-1mm}m^{-2d-1}\log m \big).
\end{eqnarray*}
As a consequence, with $\rho(d)$ defined in (\ref{Lambda0d}) and $C_{j1}$ defined in Lemma \ref{lemma46},
\begin{multline}\label{C0C1}
\frac{R_{m}}{V_{m}^{2}}=\rho(d)~+~C_{1}(-2d,\beta)~~m^{-\beta}+~O\Big(m^{-\beta-\varepsilon}+m^{-2d-1}\log m \Big)\quad (m\to \infty),\quad \mbox{with} \\
C_{1}(-2d,\beta):=2 \, \frac{c_{1}}{c_{0}}\frac 1 {C^2_{41}(-2d)}\big[C_{61}(-2d)C_{41}(-2d+\beta)-C_{61}(-2d+\beta)
C_{41}(-2d)\big],
\end{multline}
and numerical experiments proves that $ C_{1}(-2d,\beta)/c_1$ is negative for any $d \in (-0.5,0.5)$ and $\beta>0$. \\
2. Let  $-2d+\beta=1$.\\
Again with Lemma \ref{lemma46},
\begin{eqnarray*}
\frac{R_{m}}{V_{m}^{2}}&=&1-2\frac{[C_{61}(-2d)m^{\beta}+C'_{61}\frac{c_{1}}{c_{0}}log(m\pi)+C_{62}(-2d)+\frac{c_{1}}{c_{0}}C'_{62}+O(1)]}{[C_{41}(-2d)m^{\beta}+C'_{41}\frac{c_{1}}{c_{0}}log(m\pi)+C_{42}(-2d)+\frac{c_{1}}{c_{0}}C'_{42}+O(1)]}\\
&=&1-\frac{2}{C_{41}(a)}\big[C_{61}(-2d)+\big(C'_{61}\frac{c_{1}}{c_{0}}\log(m)\big)m^{-\beta}\big]\big[1-\big(\frac{C'_{41}}{C_{41}(a)}\frac{c_{1}}{c_{0}}\log(m)\big)m^{-\beta}\big]+O\big(m^{-\beta}\big)\\
&=&1-\frac{2}{C_{41}(-2d)}\Big[C_{61}(-2d)-\frac{c_{1}}{c_{0}}\big(\frac{C_{61}(-2d)C'_{41}}{C_{41}(-2d)}-C'_{61}\big)\log(m)~ m^{-\beta}\Big]+O\big(m^{-\beta}\big).
\end{eqnarray*}
As a consequence,
\begin{multline}\label{C01C02}
\frac{R_{m}}{V_{m}^{2}}=\rho(d)~+~C_{2}(-2d,\beta)m^{-\beta}~\log m +O\big(m^{-\beta}\big)\quad (m\to \infty),\quad \mbox{with}\\
C_{2}(-2d,\beta):=2\, \frac{c_{1}}{c_{0}}\frac{1}{C^2_{41}(-2d)}\Big(C'_{41}C_{61}(-2d)-C'_{61}C_{41}(-2d)\Big ),
\end{multline}
and numerical experiments proves that $ C_{2}(-2d,\beta)/c_1$ is negative for any $d \in (-0.5,0.5)$ and $\beta=1-2d$.\\
3. Let $-2d+\beta>1$.\\
Once again with Lemma \ref{lemma46}:
\begin{eqnarray*}
\frac{R_{m}}{V_{m}^{2}}&=&1-2\frac{\big[C_{61}(-2d)m^{1+2d}+C_{62}(-2d)+\frac{c_{1}}{c_{0}}C_{61}''(-2d+\beta)+\frac{c_{1}}{c_{0}}C_{62}''(-2d+\beta)m^{1+2d-\beta}+O(1)\big]}{C_{41}(-2d)m^{1+2d}\big[1+\frac{C_{42}(-2d)}{C_{41}(-2d)}m^{-2d-1}+\frac{c_{1}}{c_{0}}\frac{C_{41}''(-2d+\beta)}{C_{41}(-2d)}m^{-2d-1}+\frac{c_{1}}{c_{0}}\frac{C_{42}''(-2d+\beta)}{C_{41}(-2d)}m^{-\beta}+O(m^{-2d-1})\big]}\\
&=&1-\frac{2}{C_{41}(-2d)}\big[C_{61}(-2d)+O\big(m^{-2d-1}\big)\big]\big[1-
O\big(m^{-2d-1}\big)\big]\\
&=&1-\frac{2C_{61}(-2d)}{C_{41}(-2d)}+O\big(m^{-2d-1}\big).
\end{eqnarray*}
Note that it is not possible to specify the second order term of this expansion as in both the previous cases. As a consequence,
\begin{eqnarray}\label{C03C4}
\frac{R_{m}}{V_{m}^{2}}=\rho(d)~+~O\big(m^{-2d-1}\big)\quad (m\to \infty).
\end{eqnarray}
{\bf Step 2:} A Taylor expansion of $\Lambda(\cdot)$ around $\rho(d)$ provides:
\begin{eqnarray}
\Lambda\Big(\frac{R_{m}}{V_{m}^{2}}\Big) \simeq \Lambda\big(\rho(d)\big)+\Big[\frac{\partial\Lambda}{\partial\rho}\Big](\rho(d))\Big(\frac{R_{m}}{V_{m}^{2}}-\rho(d)\Big)+\frac{1}{2} \, \Big[\frac{\partial^{2}\Lambda}{\partial\rho^{2}}\Big](\rho(d))\Big(\frac{R_{m}}{V_{m}^{2}}-\rho(d)\Big)^2.
\end{eqnarray}
Note that numerical experiments show that $\displaystyle \Big[\frac{\partial\Lambda}{\partial \rho}\Big](\rho)>0.2$ for any $\rho\in (-1,1)$.
As a consequence, using the previous expansions of $R_{m}/V_{m}^{2}$ obtained in Step 1 and since $\E \big [IR_N(m)\big ]=\Lambda\big(R_{m}/V_{m}^{2}\big)$, then
\begin{eqnarray*}
\E \big [IR_N(m)\big ]=\Lambda_{0}(d)+\left \{ \begin{array}{ll} c_1 \, C'_1(d,\beta)\,m^{-\beta}+O\big (m^{-\beta-\varepsilon}+m^{-2d-1}\log m +m^{-2\beta} \big )& \mbox{if}~\beta<1+2d \\
c_1 \, C'_2(\beta)m^{-\beta}\log m +O(m^{-\beta})& \mbox{if}~\beta=1+2d \\
O\big (m^{-2d-1}\big ) & \mbox{if}~\beta>1+2d
\end{array}\right . ,
\end{eqnarray*}
with $C'_1(d,\beta)<0$ for all $d\in (-0.5,0.5)$ and $\beta\in  (0,1+2d)$ and $C'_2(\beta)<0$ for all $0<\beta<2$. \\
\end{proof}

\begin{proof}[Proof of Theorem \ref{cltnada}] Using Property \ref{devEIR}, if $m \simeq C\, N^\alpha$ with $C>0$ and $(1+2\beta)^{-1} \wedge (4d+3)^{-1} <\alpha<1$ then $\sqrt{N/m}\, \big (\E \big [IR_N(m)\big ] -\Lambda_{0}(d)\big ) \limiteN 0$ and it implies that the multidimensional CLT (\ref{TLC1}) can be replaced by
\begin{eqnarray}\label{TLC1bis}
\sqrt{\frac{N}{m}}\Big (IR_N(m_j)-\Lambda_{0}(d)\Big )_{1\leq j \leq p}\limiteloiN {\cal N}(0, \Gamma_p(d)).
\end{eqnarray}
It remains to apply the Delta-method with the function $\Lambda_0^{-1}$ to CLT (\ref{TLC1bis}). This is possible since the function $d \to \Lambda_0(d)$ is an increasing function such that $\Lambda_0'(d)>0$ and $\big (\Lambda_0^{-1})'(\Lambda_0(d))=1/\Lambda'_0(d)>0$ for all $d\in (-0.5,0.5)$. It achieves the proof of Theorem \ref{cltnada}.
\end{proof}

\begin{proof}[Proof of Proposition \ref{testnada}] For ease of writing we will note  $\widehat \Sigma_N $ instead of $\widehat \Sigma_N(N^\alpha)$ in the sequel.  We have $\big (\widetilde d_N(m)-\widehat d_N(j \, m) \big )_{1\leq j \leq p}=\widehat M_N \big (\widehat d_N(j \, m) -d \big )_{1\leq j \leq p}  $ with $\widehat M_N$ the orthogonal (for the Euclidian norm $\| \cdot\|_{\widehat \Sigma_N}$) projector matrix on $\big ((1)_{1\leq i \leq p}\big )^{\perp} $ (which is a linear subspace with dimension $p-1$ included in $\R^p$) in $\R^p$, {\it i.e.} $\widehat M_N=J_p(J_p'\widehat \Sigma_N^{-1} J_p)^{-1} J_p' \widehat \Sigma_N^{-1}$. Now, by denoting $\Sigma_N^{1/2}$ a symmetric matrix such as $\Sigma_N^{1/2}\Sigma_N^{1/2}=\Sigma_N$,
\begin{eqnarray*}
\|\big (\widetilde d_N(m)-\widehat d_N(j \, m) \big )_{1\leq j \leq p} \|^2_{\widehat \Sigma_N}& =& \big (\widehat d_N(j \, m) -d \big )'_{1\leq j \leq p} \widehat M_N  \widehat \Sigma_N^{-1} \widehat M_N \big (\widehat d_N(j \, m) -d \big )_{1\leq j \leq p}\\
& =& Z_N'  \widehat \Sigma_N^{1/2} \widehat M_N  \widehat \Sigma_N^{-1} \widehat M_N \widehat \Sigma_N^{1/2} Z_N \\
& = & \big (\widehat A_N Z_N\big )'\big (\widehat A_N Z_N\big )
\end{eqnarray*}
with $\widehat A_N=\Sigma_N^{-1/2} \widehat M_N \widehat \Sigma_N^{1/2}$ and $Z_N$ a random vector such as $\sqrt {N/m} Z_N \limiteloiN {\cal N}_p(0,I_p)$ from Theorem \ref{cltnada}. But we also have $\widehat A_N=\Sigma_N^{-1/2}J_p(J_p'\widehat \Sigma_N^{-1} J_p)^{-1} J_p' \widehat \Sigma_N^{-1/2}=\widehat H_N(\widehat H'_N \widehat H_N)^{-1} \widehat H'_N$ with $\widehat H_N=\Sigma_N^{-1/2}J_p$ a matrix of size $(p\times (p-1))$ with rank $p-1$ (since the rank of $J_p$ is $(p-1)$). Hence $\widehat A_N$ is an orthogonal projector to the linear subspace of dimension $p-1$ generated by the matrix $\widehat H_N$. Now 
 using Cochran Theorem (see for instance Anderson and Styan, 1982), $\sqrt {N/m} \, \widehat A_N Z_N$ is asymptotically a Gaussian vector such as $N/m \big (\widehat A_N Z_N\big )'\big (\widehat A_N Z_N\big ) \limiteloiN \chi^2(p-1)$.
\end{proof}

\noindent In Property \ref{devEIR}, a second order expansion of $\E [IR_N(m)]$ can not be specified in the case $\beta>2d+1$. In the following Property \ref{devEIR2}, we show some inequalities satisfied by $\E [IR_N(m)]$ which will be useful for obtaining the consistency of the adaptive estimator in this case.
\begin{property}\label{devEIR2}
Let  $X$ satisfy Assumption $S(d,\beta)$ with $-0.5<d<0.5$, $\beta>1+2d$. Moreover, suppose that the spectral density of $X$ satisfies Condition (\ref{condbis1}) or (\ref{condbis2}). Then there exists a constant $L>0$ depending only on $c_0, c_1, c_2,  d,\beta,\varepsilon$ such that\\
\begin{equation}\label{EIRineg}
\big |\E \big [IR_N(m)\big ]-\Lambda_0(d) \big | \geq L \,  m^{-2d-1}.
\end{equation}
\end{property}
\begin{proof}[Proof of Property \ref{devEIR2}]
Using the expansion of $J_j(a,m)$, $j=4,6$, for $a>1$ (see Lemma \ref{lemma46}) and the same computations than in Property \ref{devEIR}, we obtain:
\begin{multline*}
-\frac {2} {C_{41}^2(-2d)} \, \Big [\big (C_{62}(-2d)C_{41}(-2d)-C_{42}(-2d)C_{61}(-2d)\big)+ \frac {c_1}{c_0}\big (C''_{61}(-2d+\beta)C_{41}(-2d)-C''_{41}(-2d+\beta)C_{61}(-2d)\big)     \\
+\frac {|c_2|}{c_0}\big (C''_{61}(-2d+\beta+\varepsilon)C_{41}(-2d)+C''_{41}(-2d+\beta+\varepsilon)C_{61}(-2d)\big) \Big ] \, m^{-2d-1}(1+o(1))\\
 \leq \frac {R_m}{V_m^2} -\rho(d) \leq \\
-\frac {2} {C_{41}^2(-2d)} \, \Big [\big (C_{62}(-2d)C_{41}(-2d)-C_{42}(-2d)C_{61}(-2d)\big)+ \frac {c_1}{c_0}\big (C''_{61}(-2d+\beta)C_{41}(-2d)-C''_{41}(-2d+\beta)C_{61}(-2d)\big)     \\
-\frac {|c_2|}{c_0}\big (C''_{61}(-2d+\beta+\varepsilon)C_{41}(-2d)+C''_{41}(-2d+\beta+\varepsilon)C_{61}(-2d)\big) \Big ] \, m^{-2d-1}(1+o(1)).
\end{multline*}
Now, denote
\begin{eqnarray*}
D_0(d)&:=& C_{62}(-2d)C_{41}(-2d)-C_{42}(-2d)C_{61}(-2d)=\frac {C_{42}(-2d)C_{41}(-2d)} {48(1-2^{-1+2d})} \, \big (2^{4+2d}-5-3^{2+2d} \big ) , \\
D_1(d,\beta)&:=&C_{62}(-2d+\beta)C_{41}(-2d)-C_{42}(-2d+\beta)C_{61}(-2d)=\frac {C_{42}(-2d+\beta)C_{41}(-2d)} {128(1-2^{-1+2d})} \, \big (2^{4+2d}-5-3^{2+2d} \big ), \\
D_2(d,\beta,\varepsilon)&:=&C''_{61}(-2d+\beta+\varepsilon)C_{41}(-2d)+C''_{41}(-2d+\beta+\varepsilon)C_{61}(-2d) .
\end{eqnarray*}
Since $-0.5<d<0.5$, $2^{4+2d}-5-3^{2+2d}>0$ and $1-2^{-1+2d}>0$. Moreover, from the sign of the constants presented in Lemma \ref{lemma46}, we have $D_0(d)\neq0$ except for $d=0$, $D_1(d,\beta)\neq 0$ except for $d=2\beta$ and
$D_2(d,\beta,\varepsilon)>0$ for all $d \in (-0.5,0.5)$, $\beta>0$ and $\varepsilon>0$. 
Therefore, if $c_0, c_1, c_2,  d,\beta,\varepsilon$ are such that
\begin{eqnarray}
\label{condbis1} K_1&:=&D_0(d)+\frac {c_1}{c_0} D_1(d,\beta) - \frac {|c_2|}{c_0} D_2(d,\beta,\varepsilon)>0 \quad \\
\label{condbis2} \mbox{or}\quad K_2&:=&D_0(d)+\frac {c_1}{c_0} D_1(d,\beta) + \frac {|c_2|}{c_0} D_2(d,\beta,\varepsilon)<0.
\end{eqnarray}
and from the signs of $D_0(d)$, $D_1(d,\beta)$ and $D_2(d,\beta,\varepsilon)$, when $(d,\beta,\varepsilon)$ is fixed, these conditions are not impossible but hold following the values of $\frac {c_1}{c_0}$ and $\frac {|c_2|}{c_0} $. 
Then $\displaystyle \frac {R_m}{V_m^2} -\rho(d) \leq -\frac {K_1} {C_{41}^2(-2d)} \,  m^{-2d-1}$ or $\displaystyle \frac {R_m}{V_m^2} -\rho(d) \geq -\frac {K_2} {C_{41}^2(-2d)} \,  m^{-2d-1}$ for $m$ large enough following \eqref{condbis1} or \eqref{condbis2} holds. Then, if \eqref{condbis1} holds, since $\E [IR_N(m)]=\Lambda( \frac {R_m}{V_m^2})$, since the function $r\to \Lambda(r)$ 
is an increasing and ${\cal C}^1$  function and since $\displaystyle \E [IR_N(m)]=\Lambda \big ( \frac {R_m}{V_m^2} \big ) $  then when $m$ large enough, from a Taylor expansion,
\begin{eqnarray*}
\E [IR_N(m)]\leq \Lambda \Big ( \rho(d) -\frac {K_1} {C_{41}^2(-2d)} \,  m^{-2d-1} \Big ) & \Longrightarrow & \E [IR_N(m)]\leq \Lambda_0(d)-\frac 1 2 \Lambda'(\rho(d))\frac {K_1} {C_{41}^2(-2d)} \,  m^{-2d-1}. 
\end{eqnarray*}
Now following the same process if \eqref{condbis2} holds, we deduce inequality
(\ref{EIRineg}).
\end{proof}

\begin{proof}[Proof of Proposition \ref{hatalpha}] Let $\varepsilon>0$ be a fixed
positive real number, such that $\alpha^*+\varepsilon<1$. \\
~\\
{\bf I.} First, a bound of $\Pr(\widehat \alpha_N \leq
\alpha^*+\varepsilon)$ is provided. Indeed,
\begin{eqnarray}
\nonumber \Pr\big (\widehat \alpha_N \leq \alpha^*+\varepsilon \big
) & \geq & \Pr\Big (\widehat Q_N(\alpha^*+\varepsilon/2) \leq
\min_{\alpha \geq \alpha^*+\varepsilon~\mbox{and}~\alpha \in {\cal
A}_N}\widehat Q_N(\alpha)\Big ) \\
\nonumber & \geq & 1 - \Pr\Big (\bigcup_{\alpha \geq
\alpha^*+\varepsilon~\mbox{and}~\alpha \in {\cal A}_N}\widehat
Q_N(\alpha^*+\varepsilon/2) > \widehat Q_N(\alpha)\Big )  \\
 \label{QN} & \geq &1 -  \sum_{k =[(\alpha^*+\varepsilon)\log N]}^{\log
[N/p]}\Pr\Big (\widehat Q_N(\alpha^*+\varepsilon/2) > \widehat
Q_N\big (\frac {k}{\log N} \big)\Big ).
\end{eqnarray}
But, for $\alpha\geq \alpha^*+\varepsilon$,
\begin{multline*}
\Pr\Big (\widehat Q_N(\alpha^*+\varepsilon/2) > \widehat
Q_N(\alpha)\Big ) \\=
\Pr \Big (\big \| \big ( \widehat d_N(i\, N^{\alpha^*+\varepsilon/2})\big )_{1\leq i \leq p} - \widetilde d_N(N^{\alpha^*+\varepsilon/2})\big \|^2 _{\widehat \Sigma_N(N^{\alpha^*+\varepsilon/2})} > \big \|  \big ( \widehat d_N(i\, N^{\alpha})- \widetilde d_N(N^{\alpha})\big )_{1\leq i \leq p}\big \|^2 _{\widehat \Sigma_N(N^{\alpha}) }\Big )
\end{multline*}
with $\| X\|^2_\Omega=X'\, \Omega^{-1} \, X$.  Set $Z_N(\alpha)=\frac N  {N^\alpha} \, \big \|  \big ( \widehat d_N(i\, N^{\alpha})\big )_{1\leq i \leq p}- \widetilde d_N(N^{\alpha})\big \|^2 _{\widehat \Sigma_N(N^{\alpha}) }$. Then,
\begin{eqnarray*}
\Pr\Big (\widehat Q_N(\alpha^*+\varepsilon/2) > \widehat
Q_N(\alpha)\Big )
& = & \Pr \Big ( Z_N(\alpha^*+\varepsilon/2) >N^{\alpha-(\alpha^*+\varepsilon/2)}\, Z_N(\alpha)\Big )
\\
&\leq & \Pr \Big ( Z_N(\alpha^*+\varepsilon/2) >N^{(\alpha-(\alpha^*+\varepsilon/2))/2}\Big )+\Pr \Big ( Z_N(\alpha)< N^{-(\alpha-(\alpha^*+\varepsilon/2))/2}\Big ).
\end{eqnarray*}
From Proposition \ref{testnada}, for all $\alpha >\alpha^*$, $Z_N(\alpha) \limiteloiN \chi^2(p-1)$. As a consequence, for  $N$ large enough,
\begin{eqnarray*}
\Pr\left ( Z_N(\alpha) \leq
N^{-(\alpha-(\alpha^*+\varepsilon/2))/2} \right) 
& \leq & \frac 2 {2^{(p-1)/2}\Gamma((p-1)/2)}\cdot N^{-(\frac {p-1} 2)\frac {
(\alpha-(\alpha^*+\varepsilon/2))} 2}.
\end{eqnarray*}
Moreover, from Markov inequality and with $N$ large enough,
\begin{eqnarray*}
\Pr\left ( Z_N(\alpha^*+\varepsilon/2)
> N^{(\alpha-(\alpha^*+\varepsilon/2))/2} \right) & \leq & 2 \, \Pr\left (
\exp(
\sqrt{\chi^2(p-1}) > \exp \big ( N^{(\alpha-(\alpha^*+\varepsilon/2))/4}\big ) \right)\\
& \leq & 2 \, \E (\exp( \sqrt{\chi^2(p-1})) \, \exp \big (-
N^{(\alpha-(\alpha^*+\varepsilon/2))/4}\big ).
\end{eqnarray*}
We deduce that there exists $M_1>0$
not depending on $N$, such that for large enough $N$,
\begin{eqnarray*}
\Pr\Big (\widehat Q_N(\alpha^*+\varepsilon/2) > \widehat
Q_N(\alpha)\Big )  \leq M_1 \, \exp \big (-
N^{(\alpha-(\alpha^*+\varepsilon/2))/4}\big ).
\end{eqnarray*}
since $\E (\exp(
\sqrt{\chi^2(p-1}))<\infty$ does not depend on $N$.
Thus, the inequality (\ref{QN}) becomes, with $M_2 >0$ and for $N$ large
enough,
\begin{eqnarray}
\nonumber \Pr\big (\widehat \alpha_N \leq \alpha^*+\varepsilon \big
) & \geq &1 - M_1\,e^{-N^{\varepsilon/8}} \sum_{k =0}^{\log
[N/p]-[(\alpha^*+\varepsilon)\log N]} \exp \big (-
N^{\frac {k}{4\log N}}\big ) \\
\label{borne1} & \geq &1 - M_2 \, e^{-N^{\varepsilon/8}}.
\end{eqnarray}
{\bf II.} Secondly, a bound of $\Pr(\widehat \alpha_N \geq
\alpha^*-\varepsilon)$ can also be computed. Following the previous arguments
and notations,
\begin{eqnarray}
\nonumber \Pr\big (\widehat \alpha_N \geq \alpha^*-\varepsilon \big
) & \geq & \Pr\Big (\widehat Q_N(\alpha^*+\frac {1-\alpha^*} {2 \alpha^*}
\varepsilon) \leq \min_{\alpha \leq
\alpha^*-\varepsilon~\mbox{and}~\alpha \in {\cal
A}_N}\widehat Q_N(\alpha)\Big ) \\
\label{QN2} & \geq &1 -  \sum_{k =2}^{[(\alpha^*-\varepsilon) \log
N]+1}\Pr\Big (\widehat Q_N(\alpha^* +\frac {1-\alpha^*} {2 \alpha^*}
\varepsilon) > \widehat Q_N\big (\frac {k}{\log N} \big)\Big ),
\end{eqnarray}
and as above, with $Z_N(\alpha)=\frac N  {N^\alpha} \, \big \|  \big ( \widehat d_N(i\, N^{\alpha})-\widetilde d_N(N^{\alpha})\big )_{1\leq i \leq p} \big \|^2 _{\widehat \Sigma_N(N^{\alpha}) }$,
\begin{eqnarray}
\label{TLC7}  \Pr\Big (\widehat Q_N(\alpha^*+\frac {1-\alpha^*} {2 \alpha^*}
\varepsilon) > \widehat Q_N(\alpha)\Big ) =  \Pr\Big ( Z_N(\alpha^*+\frac {1-\alpha^*} {2 \alpha^*}
\varepsilon)
> N^{\alpha-(\alpha^*+\frac {1-\alpha^*} {2 \alpha^*}
\varepsilon)} Z_N(\alpha) \Big ).
\end{eqnarray}
$\bullet$ if $\beta\leq 2d+1$, with $\alpha <\alpha^*=(1+2\beta)^{-1}$, from Property \ref{devEIR} and with $C \neq 0$, for $1\leq i \leq p$,
\begin{multline}\label{expan}
\sqrt{ \frac N {N^\alpha}} \, \big (\E \big [IR(i\, N^\alpha)\big ]-\Lambda_0(d)\big ) \simeq  C\,i^{-(1-\alpha^*)/2\alpha^*} \,  N^{(\alpha^*-\alpha)/2\alpha^*} (\log N)^{\1_{\beta=2d+1}} \\
\Longrightarrow \sqrt{ \frac N {N^\alpha}} \, \big( \Lambda_0^{-1}(\E \big [IR(i\, N^\alpha)\big ])-d \big ) \simeq  C'\,i^{-(1-\alpha^*)/2\alpha^*} \,  N^{(\alpha^*-\alpha)/2\alpha^*} (\log N)^{\1_{\beta=2d+1}} 
\end{multline}
with $C'\neq 0$, since $\Lambda_0(d)>0$ for all $d\in (-0.5,0.5)$. We deduce:
\begin{eqnarray*}
 && \sqrt{ \frac N {N^\alpha}} \, \Big (  \widehat d_N(i\, N^{\alpha})-d\Big )_{1\leq i \leq p} \simeq  C''\,   N^{(\alpha^*-\alpha)/2\alpha^*} (\log N)^{\1_{\beta=2d+1}}\big (i^{-(1-\alpha^*)/2\alpha^*} \big)_{1\leq i\leq p}+ \big ( \widehat  \varepsilon_N(i\, N^\alpha)\big )_{1\leq i\leq p}, 
\end{eqnarray*}
with $C''\neq 0$ and $\big ( \widehat  \varepsilon_N(i\, N^\alpha)\big )_{1\leq i\leq p} \limiteloiN {\cal N}\Big(0, (\Lambda'_0(d))^{-2}\, \Gamma_p(d)\Big )$ 
from Proposition \ref{cltnada}. Now from the definition of $\widetilde d_N(N^{\alpha})$, we have $\big ( \widehat d_N(i\, N^{\alpha})-\widetilde d_N(N^{\alpha})\big )_{1\leq i \leq p}=\widehat M_N \big ( \widehat d_N(i\, N^{\alpha})-d\big )_{1\leq i \leq p}$ with $\widehat M_N$ the orthogonal  projector matrix on $(1)^{\perp}_{1\leq i \leq p}$. \\
As a consequence, for $\alpha<\alpha^*-\varepsilon$ and with the inequality $\|a-b\|^2\geq \frac 1 2 \|a \|^2-\|b\|^2$, 
\begin{eqnarray*}
\nonumber Z_N(\alpha) &\geq &\frac 1 2  (C'')^2\,  N^{\frac {\alpha^*-\alpha}{\alpha^*}} (\log^2 N)^{\1_{\beta=2d+1}}  \Big \| \widehat M_N \big (i^{-\frac {1-\alpha^*}{2\alpha^*}}\big )_{1\leq i\leq p} \Big \|_{\widehat \Sigma_N(N^\alpha)} - \|\widehat M_N  \widehat \varepsilon_N(i\, N^\alpha))\|^2_{\widehat \Sigma_N(N^\alpha)}.
\end{eqnarray*}
Now, it is clear that $ \|\widehat M_N  \widehat \varepsilon_N(i\, N^\alpha))\|^2_{\widehat \Sigma_N(N^\alpha)} \leq \|  \widehat \varepsilon_N(i\, N^\alpha))\|^2_{\widehat \Sigma_N(N^\alpha)} \leq C_1$ when $N$ large enough, with $C_1>0$ not depending on $N$. Moreover the vector $\big (i^{-\frac {1-\alpha^*}{2\alpha^*}}\big )_{1\leq i\leq p} $
is not in the  subspace $(1)_{1\leq i \leq p}$ and therefore $\Big \| \widehat M_N \big (i^{-\frac {1-\alpha^*}{2\alpha^*}}\big )_{1\leq i\leq p} \Big \|_{\widehat \Sigma_N(N^\alpha)}\geq C_2$ for $N$ large enough with $C_2>0$. We deduce that there exists $D>0$ such that for $N$ large enough and $\alpha<\alpha^*-\varepsilon$,
\begin{eqnarray}\label{egal}
\nonumber Z_N(\alpha) &\geq & D \,   N^{\frac {\alpha^*-\alpha}{\alpha^*}} (\log^2 N)^{\1_{\beta=2d+1}}.
\end{eqnarray}
Therefore, since $ N^{\frac {\alpha^*-\alpha}{\alpha^*}} \limiteN \infty $ when $\alpha<\alpha^*-\varepsilon$,
$$
\Pr \Big ( Z_N(\alpha) \geq \frac 1 2  D \,   N^{\frac {\alpha^*-\alpha}{\alpha^*}}  \Big ) \limiteN 1.
$$
Then, the relation (\ref{TLC7}) becomes for $\alpha<\alpha^*-\varepsilon$  and $N$ large enough,
\begin{eqnarray*}
\Pr\Big (\widehat Q_N(\alpha^*+\frac {1-\alpha^*} {2 \alpha^*}
\varepsilon) > \widehat Q_N(\alpha)\Big ) &\leq & \Pr\Big (
\chi^2 (p-1) \geq \big (\frac 1 2\,  D \,   N^{\frac {\alpha^*-\alpha}{\alpha^*}}  \big ) \, N^{\alpha-(\alpha^*+\frac {1-\alpha^*} {2 \alpha^*}
\varepsilon)}\Big ) \\
&\leq & \Pr\Big (
\chi^2 (p-1) \geq  \frac D 2  \, N^{\frac {1-\alpha^*} {2 \alpha^*} (2(\alpha^*-\alpha)-
\varepsilon)}\Big )  \\
& \leq & M_2 \, N^{-(\frac {p-1} 2)\frac {1-\alpha^*} {2
\alpha^*}\, \varepsilon},
\end{eqnarray*}
with $M_2 >0$, because
$\frac {1-\alpha^*} {2 \alpha^*}(2(\alpha^*-\alpha)-
\varepsilon)\geq \frac {1-\alpha^*} {2 \alpha^*}\varepsilon$ for
all $\alpha \leq \alpha^*-\varepsilon$. Hence, from the inequality
(\ref{QN2}), for large enough $N$,
\begin{eqnarray}
\label{borne2} \Pr\big (\widehat \alpha_N \geq \alpha^*-\varepsilon
\big ) \geq 1 -M_2 \, \log N \, N^{-(p-1)\frac
{1-\alpha^*} { 4\alpha^*}\, \varepsilon}.
\end{eqnarray}
$\bullet$ if $\beta > 2d+1$, with $\alpha<\alpha^*=(4d+3)^{-1}$ and from Property \ref{devEIR2}, we obtain an inequality instead of (\ref{expan}): 
$$
\Big | \Lambda_0^{-1} \big (\E [IR_N(m)] \big )-   d\Big |\geq \frac 1 2 (\Lambda_0(d))^{-1} L \,  m^{-2d-1}
$$
since the function $x \mapsto \Lambda_0^{-1}(x)$ is an increasing an ${\cal C}^1$ function, using a Taylor expansion. Therefore for $1\leq i \leq p$,
\begin{eqnarray}\label{inegal}
\sqrt{ \frac N {N^\alpha}} \, \Big |\Lambda_0^{-1} \big ( \E \big [IR(i\, N^\alpha)\big ]\big )-d\Big | \geq  \frac 1 2 (\Lambda_0(d))^{-1} L\,i^{-(1-\alpha^*)/2\alpha^*} \,  N^{(\alpha^*-\alpha)/2\alpha^*}.
\end{eqnarray}
Now, as previously and with the same notation, 
\begin{multline}\label{expan2}
\big ( \widehat d_N(i\, N^{\alpha})-\widetilde d_N(N^{\alpha})\big )_{1\leq i \leq p}\simeq    \widehat M_n  \big (\Lambda_0^{-1} \Big ( \E \big [IR(i\, N^\alpha)\big ]\big )-d \Big ) _{1\leq i\leq p}
+\widehat M_n  \big ( \widehat \varepsilon_N(i\, N^\alpha)\big )_{1\leq i\leq p}.
\end{multline}
Now plugging \eqref{inegal} in \eqref{expan2} and following the same steps of the proof in the case $\beta\leq 2d+1$, the same kind of bound (\ref{borne2}) can be obtained. \\
~\\
Finally, the inequalities (\ref{borne1}) and (\ref{borne2}) imply that
$\displaystyle{~\Pr\big (|\widehat \alpha_N - \alpha^*| \geq
\varepsilon \big ) \limiteN 0}$.
\end{proof}

\begin{proof}[Proof of Theorem \ref{tildeD}]
The results of Theorem \ref{tildeD} can be easily deduced from Theorem \ref{cltnada} and Proposition \ref{hatalpha} (and its proof) by using conditional probabilities.
\end{proof}
\begin{proof}[Proof of Proposition \ref{testada}]
{Proposition} {\ref{testada}} can be deduced from {Theorem} {\ref{tildeD}} using the same kind of proof than in Proposition \ref{testnada} and conditional distributions.
\end{proof}

\begin{lemma}\label{lemma46}
For $j=4,6$, denote
\begin{equation}\label{Jj}
J_j(a,m):=\int_{0}^{\pi}x^{a}\frac{\sin^{j}(\frac{mx}{2})}{\sin^{2}(\frac{x}{2})}dx.
\end{equation}
Then, we have the following expansion when $m \to \infty$:
\begin{eqnarray*}
1.& \mbox{if $-1<a<1$,} &
J_j(a,m)=C_{j1}(a)\,m^{1-a}+C_{j2}(a)+O\big (m^{-1-(a\wedge0)}\big); \\
2. &\mbox{if $a=1$,} &
J_j(a,m)=C_{j1}'\,\log(m)+C_{j2}'+O\big (m^{-1}\big); \\
3. &\mbox{if $a>1$,} &
J_j(a,m)=C_{j1}''(a)+O\big (m^{1-a}+m^{-2} \big),
\end{eqnarray*}
where constants $C_{j1}(a)$, $C_{j2}(a)$, $C'_{j1}(a)$, $C'_{j2}(a)$ and $C''_{j1}(a)$ are specified in the following proof.
\end{lemma}
\begin{proof}[Proof of Lemma \ref{lemma46}]
1. let $-1<a<1$. \\
We begin with the expansion of $\displaystyle J_4(a,m)$. First, decompose $J_4(a,m)$ as follows
\begin{eqnarray} \label{Jam}
J_4(a,m)=2^{a+1}\int_{0}^{\frac{\pi}{2}}y^{a}\sin^{4}(my)\Big [\frac{1}{\sin^{2}(y)}-\frac{1}{y^{2}}\Big ]dy+\int_{0}^{\pi}\frac {x^{a}}{(\frac{x}{2})^{2}} \sin^{4}(\frac{mx}{2})dx.
\end{eqnarray}
Using integrations by parts and $\sin^{4}(\frac{x}{2})=\sin^{2}(\frac{x}{2})-\frac{1}{4}\sin^{2}(x)= \frac 1 8 \big ( 3-4\cos(y)+\cos(2y) \big )$, we obtain for $m\to \infty$:
\begin{eqnarray*}
\int_{0}^{\pi}\frac {x^{a}}{(\frac{x}{2})^{2}} \sin^{4}(\frac{mx}{2})dx&=& 4\,m^{1-a} \Big ( (1-\frac{1}{2^{1+a}})\int_{0}^{\infty}\frac{\sin^{2}(\frac{y}{2})}{y^{2(\frac{1-a}{2})+1}}dy- \frac 1 8 \int_{m\pi}^{\infty}y^{a-2}\big ( 3-4\cos(y)+\cos(2y) \big )dy \Big ) \\
& =& \frac{\pi(1-\frac{1}{2^{1+a}})}{(1-a)\Gamma(1-a)\sin(\frac{(1-a)\pi}{2})} \,m^{1-a} - 3 \, \frac{1}{2(1-a)}\pi^{a-1} + O(m^{-1})
\end{eqnarray*}
where the left right side term of the last relation is obtained by integration by parts and the left side term is deduced from the following relation  (see Doukhan {\it et al.} 2003, p. 31)
\begin{eqnarray}\label{taqqu}
\int_{0}^{\infty}y^{-\alpha}~\sin(y)~dy~=~\frac{1}{2}~\frac{\pi}{\Gamma(\alpha)\sin(\pi (\frac{\alpha}{2}))}\quad \mbox{for $0<\alpha<2$.}
\end{eqnarray}
Moreover, with the linearization of $\sin^4 u$ and Taylor expansions $
\frac{1}{\sin^{2}(y)}-\frac{1}{y^{2}}\underset{y\rightarrow 0}{~\sim~}~\frac{1}{3}\quad\mbox{and}\quad
\frac{1}{y^{3}}-\frac{\cos(y)}{\sin^{3}(y)}~\underset{y\rightarrow 0}{\sim}~\frac{y}{15}
$,
\begin{eqnarray}\label{J0e}
2^{a+1}\int_{0}^{\frac{\pi}{2}}y^{a}\sin^{4}(my)\Big [\frac{1}{\sin^{2}(y)}-\frac{1}{y^{2}}\Big ]dy=~3~\frac{2^{a+1}}{8} \int_{0}^{\frac{\pi}{2}}y^{a}[\frac{1}{\sin^{2}(y)}-\frac{1}{y^{2}}]dy+O\big (m^{-1-(a\wedge 0)}\big ).
\end{eqnarray}
Finally, by replacing this expansion in (\ref{Jam}), one deduces
\begin{multline}\label{Vm4}
J_4(a,m)=\int_{0}^{\pi}x^{a}\frac{\sin^{4}(\frac{mx}{2})}{\sin^{2}(\frac{x}{2})}dx=C_{41}(a)\,m^{1-a}+C_{42}(a)+O\big (m^{-1-(a\wedge0)}\big)\quad (m\to \infty),
\mbox{with}\\~ C_{41}(a):=\frac{\pi(1-\frac{1}{2^{1+a}})}{(1-a)\Gamma(1-a)\sin(\frac{(1-a)\pi}{2})}~~\mbox{and}~C_{42}(a):=  \frac{3}{2^{2-a}}\, \int_{0}^{\frac{\pi}{2}}y^{a}[\frac{1}{\sin^{2}(y)}-\frac{1}{y^{2}}]dy-\frac{3}{2(1-a)}\pi^{a-1}.
\end{multline}
Note that $C_{41}(a)> 0$ and $C_{42}(a)< 0$ for all $0<a<1$, $C_{42}(a)> 0$ for all $-1<a<0$, $C_{42}(0)=0$.\\
~\\
A similar expansion procedure of $\displaystyle J_6(a,m)$ with $\sin^{6}(\frac{mx}{2})$ instead of $\sin^{4}(\frac{mx}{2})$ can be provided. 
As previously with  $\sin^{6}(\frac{y}{2})=\frac 1 {32 } \big (10-15\cos(y)+6\cos(2y)-\cos(3y)\big ) $, when $m\to \infty$,
\begin{multline*}
J_6(a,m)=C_{61}(a)\,m^{1-a}+C_{62}(a)+O\big (m^{-1-(a\wedge0)}\big),\\
\mbox{with}~ C_{61}(a):=\frac{\pi(15+3^{1-a}-2^{1-a}6)}{16(1-a)\Gamma(1-a)\sin(\frac \pi{2} (1-a))}~\mbox{and}~C_{62}(a):= \frac{5}{6} \, C_{42}(a).
\end{multline*}
Moreover it is clear that $C_{61}(a)> 0$.\\
~\\
2. let $a=1$. \\
When $m\to \infty$ we obtain the following expansion:
\begin{eqnarray*}
\int_{0}^{\pi}\frac{x\sin^{4}(\frac{mx}{2})}{\sin^{2}(\frac{x}{2})}dx
&=&\frac 1 2 \Big (\int_{0}^{m\pi}\frac{\sin (2x)-2x}{2x^2}dx-4\int_{0}^{m\pi}\frac{\sin (x)-x}{x^2}dx\Big ) +4\int_{0}^{\frac{\pi}{2}}y\sin^{4}(my)\Big(\frac{1}{\sin^{2}(y)}-\frac{1}{y^{2}}\Big)dy
\end{eqnarray*}
But,
$$
\int_{0}^{m\pi}\frac{\sin (2x)-2x}{2x^2}dx-4\int_{0}^{m\pi}\frac{\sin (x)-x}{x^2}dx=\frac 3 2 \Big ( \log(m\pi)+\int_1 ^\infty \frac {\sin y}{y^2}dy+\int_{0}^{1}\frac{\sin y-y}{y^2}dy \Big )+O(m^{-1}).
$$
Moreover from previous computations (see the case $a<1$),
$$
\int_{0}^{\frac{\pi}{2}}y\sin^{4}(my)\Big(\frac{1}{\sin^{2}(y)}-\frac{1}{y^{2}}\Big)dy=\frac 3 8 \, \int_{0}^{\frac{\pi}{2}}y\Big(\frac{1}{\sin^{2}(y)}-\frac{1}{y^{2}}\Big)dy+O(m^{-1}).
$$
As a consequence, when $m\to \infty$,
\begin{multline*}
\int_{0}^{\pi}\frac{x\sin^{4}(\frac{mx}{2})}{\sin^{2}(\frac{x}{2})}dx=C'_{41}\,\log(m)+C'_{42}+O\big (m^{-1}\big), \quad
\mbox{with} \quad C'_{41}:=\frac{3}{2}\quad \mbox{and}~\\
C'_{42}:= \frac{3}{2}\Big ( \log(\pi)+\int_{0}^{\frac{\pi}{2}}y\Big(\frac{1}{\sin^{2}(y)}-\frac{1}{y^{2}}\Big)dy+\int_1 ^\infty \frac {\sin y}{y^2}dy+\int_{0}^{1}\frac{\sin y-y}{y^2}dy \Big ).
\end{multline*}
Note that $C'_{41}>0$ and $C'_{42}\simeq 2.34 >0$. \\
~\\
In the same way , we obtain the following expansions when $m\to \infty$,
\begin{multline*}
\int_{0}^{\pi}\frac{x\sin^{6}(\frac{mx}{2})}{\sin^{2}(\frac{x}{2})}dx=C'_{61}\,\log(m)+C'_{62}+O\big (m^{-1}\big)\quad
\mbox{with} \quad C'_{61}:=\frac{5}{4}\quad \mbox{and}  \\
C'_{62}:= \frac{5}{4}\log(\pi)+ \frac{5}{4} \int_{0}^{\frac{\pi}{2}}y\Big(\frac{1}{\sin^{2}(y)}-\frac{1}{y^{2}}\Big)dy+\frac{1}{8}\int_{1}^{\infty}\frac{1}{y}\Big(-\cos (3y)+6\cos(2y)-15\cos(y)\Big)dy+4\int_{0}^{1}\frac{1}{y}\sin^{6}(\frac{y}{2})dy.
\end{multline*}
Note again that $C'_{61}>0$ and numerical experiments show that $C'_{62}> 0$. \\
~\\
3. Let $a>1$. Then, with the linearization of $\sin^4(u)$,
\begin{eqnarray}\label{linsin4}
\nonumber \int_{0}^{\pi}\frac{x^{a}\sin^{4}(\frac{mx}{2})}{\sin^{2}(\frac{x}{2})}dx&=& \frac{3}{8}\int_{0}^{\pi}\frac{x^{a}}{\sin^{2}(\frac{x}{2})}dx-\frac{1}{2}\int_{0}^{\pi}\frac{x^{a}}{\sin^{2}(\frac{x}{2})}\cos(mx)dx+\frac{1}{8}\int_{0}^{\pi}\frac{x^{a}}{\sin^{2}(\frac{x}{2})}\cos(2mx)dx \\
\label{cas>1} & =& C''_{41}(a)+\frac{1}{m}\int_{0}^{\pi}\Big(\frac{\sin(mx)}{2}-\frac{\sin(2mx)}{16}\Big)\Big(g(x)+h(x)\Big)dx,
\end{eqnarray}
with: $\displaystyle g(x)=\Big(\frac{ax^{a-1}}{\sin^{2}(\frac{x}{2})}-4ax^{a-3}\Big)-\Big(\frac{x^{a}\cos(\frac{x}{2})}{\sin^{3}(\frac{x}{2})}-8x^{a-3}\Big)$ and $ h(x)=(4a-8)x^{a-3}$. \\
~\\
First, if $1<a $, with an integration by parts,
\begin{eqnarray}
\frac{1}{m}\int_{0}^{\pi}\big(\frac{\sin(mx)}{2}-\frac{\sin(2mx)}{16}\big)h(x)dx 
\label{h13} &=&O\big(m^{1-a} + m^{-2}\big).
\end{eqnarray}
Moreover,
\begin{multline*}
\frac{1}{m}\int_{0}^{\pi}\Big(\frac{\sin(mx)}{2}-\frac{\sin(2mx)}{16}\Big)g(x)dx 
\nonumber  \\= \big (\frac 1 {32} -\frac {(-1)^m} 2 \big) \big (a\pi^2-4a+8\big)\pi^{a-3}\frac{1}{m^2}-\frac{1}{m^{2}}\int_{0}^{\pi}\big(-\frac{\cos(mx)}{2}+\frac{\cos(2mx)}{32}\big)\, g'(x)dx
\end{multline*}
since $g(x)\underset{x=0^{+}}\sim\frac{a}{3}~x^{a-1}$ and
$g'(x)\underset{x=0^{+}}\sim\frac{a(a-1)}{3}x^{a-2}$. Therefore, if $1<a$,
\begin{eqnarray*}
\frac{1}{m}\int_{0}^{\pi}\Big(\frac{\sin(mx)}{2}-\frac{\sin(2mx)}{16}\Big)g(x)dx 
\label{g3} &=&O\big(m^{-2}\big).
\end{eqnarray*}
In conclusion, for $1<a$ we deduce,
\begin{multline*}\label{Vm4''1}
\int_{0}^{\pi}\frac{x^{a}\sin^{4}(\frac{mx}{2})}{\sin^{2}(\frac{x}{2})}dx=C''_{41}(a)+O\big(m^{1-a} + m^{-2}\big) \quad \mbox{with}\quad  C''_{41}(a):=\frac{3}{8}\int_{0}^{\pi}\frac{x^{a}}{\sin^{2}(\frac{x}{2})}dx>0.
\end{multline*}
Similarly, for $1<a$ we deduce,
\begin{multline*}
\int_{0}^{\pi}\frac{x^{a}\sin^{6}(\frac{mx}{2})}{\sin^{2}(\frac{x}{2})}dx=C''_{61}(a)+O\big(m^{1-a} + m^{-2}\big) \quad
\mbox{with} \quad C''_{61}(a):=\frac{5}{16}\int_{0}^{\pi}\frac{x^{a}}{\sin^{2}(\frac{x}{2})}dx=\frac 5 6 \, C''_{41}(a)>0.
\end{multline*}
\end{proof}
\medskip
\noindent
{\bf Acknowledgments.} The authors are grateful to both the referees for their very careful reading and many relevant
suggestions and corrections that strongly improve the content and
the form of the paper.

\bibliographystyle{amsalpha}

\end{document}